\documentclass[preprint,12pt]{elsarticle}

\makeatletter
\def\ps@pprintTitle{%
 \let\@oddhead\@empty
 \let\@evenhead\@empty
 \def\@oddfoot{\centerline{\thepage}}%
 \let\@evenfoot\@oddfoot}
\makeatother

\usepackage{amsmath} \usepackage{amsthm} \usepackage{amsfonts} \usepackage{amssymb}
\usepackage{verbatim}

 \newcommand{\ep}{\varepsilon}
\newcommand{\B}{\mathcal{B}}

\newcommand{\R}{\mathbb{R}}

\newcommand{\dd}{\,\mathrm{d}}
\newcommand{\x}{x}

\newcommand{\EW}{\mathbb{E}}

\newcommand{\Var}{\text{Var}}

\newcommand{\y}{y}

\newtheorem{remark}{Remark}

\newtheorem{theorem}{Theorem}

\theoremstyle{remark}

\usepackage{amsmath} \usepackage{amsthm} \usepackage{amsfonts} \usepackage{amssymb}
\usepackage{bbm}
\usepackage{mathtools} 
\usepackage{booktabs}
\usepackage[all, knot]{xy}
\usepackage{tikz}
 \usepackage{tkz-euclide}
\usetikzlibrary{calc,decorations.markings}
\usetikzlibrary{positioning}
\usetikzlibrary{trees}
\usetikzlibrary{decorations.pathmorphing}
\usetikzlibrary{decorations.markings}
\usetikzlibrary{decorations.pathreplacing,decorations.pathmorphing}
\usepackage{graphicx}
\usepackage{graphics} 
\usepackage{epsfig} 
\usepackage{psfrag}
\usepackage{amsbsy} 

\begin{document}

\begin{frontmatter}

\title{A Partially Reflecting Random Walk on Spheres Algorithm for Electrical Impedance Tomography}

\author[label1]{Sylvain Maire}\ead{maire@univ-tln.fr}\author[label2]{Martin Simon\fnref{label3}}\ead{simon@math.uni-mainz.de}
\address[label1]{Laboratoire LSIS Equipe Signal et Image
Universit\'{e} du Sud Toulon-Var, AV. Georges Pompidou BP 56, 83162 La Valette du Var Cedex, France}
\address[label2]{Institute of Mathematics, Johannes Gutenberg University, 55099 Mainz, Germany}
\fntext[label3]{The second author was supported by DFG grant HA 2121/8 -1 583067.}

\begin{abstract}
In this work, we develop a probabilistic estimator for the voltage-to-current map arising in electrical impedance tomography. This novel so-called \emph{partially reflecting random walk on spheres estimator} enables Monte Carlo methods to compute the voltage-to-current map in an embarrassingly parallel manner, which is an important issue with regard to the corresponding inverse problem. Our method uses the well-known random walk on spheres algorithm inside subdomains where the diffusion coefficient is constant and employs replacement techniques motivated by finite difference discretization to deal with both mixed boundary conditions and interface transmission conditions. We analyze the global bias and the variance of the new estimator both theoretically and experimentally. In a second step, the variance is considerably reduced via a novel control variate conditional sampling technique. 
\end{abstract}


\end{frontmatter}


%%%%%%%%%%%%%%%%%%%%%%
\section{Introduction}
The mathematical formulation of static electrical impedance tomography (EIT) leads to a nonlinear and ill-posed inverse problem, which is unstable with respect to measurement and modeling errors. Namely the reconstruction of the real-valued conductivity $\kappa$ in the elliptic \emph{conductivity equation}
\begin{equation}\label{eqn:con}
\nabla\cdot(\kappa\nabla u)=0\quad\text{in }D
\end{equation}
from boundary measurements of the electric potential $u$ and the corresponding current on the boundary of a bounded, convex domain $D\subset\mathbb{R}^d$, $d=2,3$, with piecewise smooth boundary $\partial D$ and connected complement. Due to the limited capabilities of static EIT, many practical applications focus on the detection of conductivity anomalies in a known background conductivity rather than conductivity imaging, cf., e.g., Pursiainen \cite{Pursiainen} and the recent work \cite{Simon} by the second author. In this work, we consider such an \emph{anomaly detection problem}, where a perfectly conducting inclusion occupies a region $T$ inside the domain $D$.  A possible practical application modeled by this setting is breast cancer detection, where the electric conductivity of high-water-content tissue, such as malignant tumors, is approximately one order of magnitude higher than the conductivity of low-water-content tissue, such as fat, which is the main component of healthy breast tissue, cf. \cite{Isaacson1}.

The most accurate mathematical forward model for real-life impedance tomography is the \textit{complete electrode model} (CEM), cf. \cite{Somersalo}, 
where the electric potential $u$ is assumed to satisfy the Robin boundary condition
\begin{equation}\label{eqn:Robin}
z\mathbb{{\nu}}\cdot\kappa\nabla u|_{\partial D}+u|_{\partial D}=\phi\quad\text{on}\ \partial D.
\end{equation}
Here ${\nu}$ denotes the outer unit normal vector on $\partial D$ and the positive constant $z$ is the so-called \textit{contact impedance} which accounts for electrochemical effects at the electrode-skin interface. Given the full Robin-to-Neumann map
\begin{equation*}
R_{z,\kappa}:\phi\mapsto{\nu}\cdot\kappa\nabla u|_{\partial D},
\end{equation*}
that maps the potential on the boundary to the corresponding current across the boundary, this knowledge uniquely determines $z$ and is hence equivalent to the knowledge of the Dirichlet-to-Neumann map. In this case, uniqueness of solutions to the inverse conductivity problem for isotropic conductivities has been proved under various assumptions on both, spatial dimension and regularity of the conductivity, cf., e.g., the works by Astala and P\"aiv\"arinta \cite{Astala} for $d=2$ and Haberman and Tataru \cite{Haberman} for $d=3$.

Notice that the operator $R_{z,\kappa}$ corresponds to idealized measurements on the whole boundary $\partial D$. In practice, however, only a finite number of finite-sized electrodes is available and thus only incomplete and noisy measurements of the Robin-to-Neumann map can be obtained. Given such discrete \emph{voltage-to-current maps}, the use of a regularization strategy is mandatory because of the severe ill-posedness of the inverse problem, cf. Alessandrini \cite{Alessandrini}. In statistical inversion theory, the inverse problem is therefore formulated in the framework of Bayesian statistics, that is, all the variables included in the mathematical model are treated as random variables. The solution to the statistical inverse problem is then given by the posterior probability distribution of the unknown parameters conditioned on the measured data, see e.g., \cite{KaipioSomersalo, Kaipioetal, Jin}.
Computing the \textit{conditional mean} estimate as well as common spread estimates from the posterior density leads to high-dimensional integration problems and Markov chain Monte Carlo (MCMC) techniques are usually employed for this task. However, each sampling step in such an algorithm requires solving the forward problem (\ref{eqn:con}), (\ref{eqn:Robin}) numerically so that the computation time can easily become excessive. This effect is amplified by the fact that the Robin boundary condition (\ref{eqn:Robin}) leads to singularities of the solution $u$ at the end points of the electrodes such that numerical approximations, both via finite element and boundary element methods, require very fine discretization. 

In this work, we are concerned with the forward problem of EIT. More precisely, we develop a probabilistic estimator for the voltage-to-current map which has potential to overcome the aforementioned drawback if it is used on massively parallel hardware, such as GPUs, within the so-called Bayesian \textit{modeling error approach}, cf. Kaipio and Somersalo \cite{KaipioSomersalo}. The main advantage of the proposed method, beside its inherent parallel scalability, comes from the fact that the error estimates required for the Bayesian modeling error approach may be computed adaptively and on the fly at almost no additional computational cost. On top of that, our approach is  well suited for uncertainty quantification in problems with random parameters. 

Due to the advent of multicore computing architectures, probabilistic estimators for the numerical solution of 
boundary value problems for PDE in three or more dimensions have become a valuable alternative to deterministic methods. This is particularly true, when one needs to compute the solution at only a few points, or when moderate accuracy is sufficient. For instance in biophysical applications, where the linearized Poisson-Boltzmann equation must be solved, efficient probabilistic numerical algorithms have been developed recently, see e.g. \cite{Mascagni, Bossyetal1}.
However, in contrast to these works, the derivation of a probabilistic estimator for the voltage-to-current map corresponds to the approximation of paths of the \emph{partially reflecting Brownian motion}, cf. \cite{Grebenkov}, rather than the killed Brownian motion. The partially reflecting Brownian motion behaves like the standard Brownian motion inside the domain and it is prevented from leaving the domain either by absorption or by instantaneous reflection. Under quite general assumptions, a Feynman-Kac type representation formula in terms of the boundary local time process of the partially reflecting Brownian motion for the electric potentials in EIT was recently obtained by Piiroinen and the second author in \cite{Piiroinen}. 
It is, however, well-known that direct simulation of the underlying Lebesgue-Stieltjes integrals with respect to the boundary local time process is quite a difficult task, see e.g. \cite{Constantini,Gobet,Gobet1}. To be precise, the first order convergence obtained by Gobet's \textit{half-space approximation} scheme \cite{Gobet1} is currently the state of the art.

In this work we propose a different approach, namely we discretize with respect to space by expressing the unknown electrical potential as the expectation of some auxiliary random variable obtained via a local finite difference discretization. This yields a novel second order space discretization scheme.
A similar technique, using a first order approximation, was first introduced by Mascagni and Simonov in \cite{Mascagni} in the context of simulation of diffusion processes in discontinuous media. Also for the simulation of diffusion processes in discontinuous media, second order schemes were proposed and analyzed by Bossy et al. \cite{Bossyetal1} and by Lejay and the first author \cite{LejayMaire}. The idea to use a local finite difference discretization for the simulation of the boundary behavior of reflecting diffusion processes was introduced recently by the first author and Tanr\'{e} \cite{MaireTanre} and further developed by the first author and Nguyen \cite{MaireNguyen}. Other, related schemes were defined by Lejay and Pichot \cite{Lejay} and Lejay and the first author \cite{LejayMaire1}. The main advantage of the method proposed in this work, in comparison to the aforementioned works, lies in the fact that the variance of our estimator is greatly reduced due to a combined control variates conditional sampling technique. Therefore, we expect the method to become a valuable alternative to established deterministic methods for the problem at hand.

The rest of the paper is structured as follows: We start in Section 2 by describing briefly the modeling of electrode measurements and the anomaly detection problem in EIT. In Section 3 we recall the basic idea of the random walk on spheres (RWOS). Subsequently in Section 4 we introduce the novel partially reflecting random walk on spheres estimator and in Section 5 we describe how it can be used to approximate electrode measurements. Then in Section 6 the variance reduction technique is explained. Section 7 generalizes the partially reflecting random walk on spheres estimator to problems with layered conductivities as well as problems with random parameters. In Section 8 we present numerical examples to illustrate the efficiency of our algorithm. Finally, we conclude with a brief summary of our results and comment on directions for future research.

%%%%%%%%%%%%%%%%%%%%%%

\section{Modeling of electrode measurements}\label{sec:2}
Let us briefly recall both, the modeling of electrode measurements in EIT and the anomaly detection problem. Consider the time-harmonic Maxwell's equations, to be precise, Faraday's law and Amp\`{e}re's law 
\begin{equation*}
 \mathrm{curl}{ E}=i\omega\mu {H},\quad \mathrm{curl} {H}=-(i\omega\ep-\kappa){E},
\end{equation*}
where $\omega$ is the frequency, $\mu$ the magnetic permeability and $\ep$ the electric permittivity. Notice that the physically relevant electric and magnetic field are given by the real parts $\Re({E}({x})e^{i\omega t})$ and $\Re({H}({x})e^{i\omega t})$, respectively. Now let us specialize on the static, respectively quasi-static case, i.e. direct input currents, respectively low frequencies $\omega$. Then the imaginary part of the electrical admittivity $i\omega\ep-\kappa$ becomes negligible as well as the term $i\omega\mu {H}$. It can indeed be shown that the Maxwell system is approximated by
\begin{equation}\label{eqn:Maxwell}
 \mathrm{curl} {E}=0,\quad \mathrm{curl} {H}=\kappa {E},
\end{equation}
see e.g. \cite{Cheney}. In particular the electric field must be a gradient field ${E}=-\nabla u$ for the scalar electric potential $u$. Substitution of this expression into the second equation in (\ref{eqn:Maxwell}) and taking the divergence finally yields the conductivity equation (\ref{eqn:con}).

In anomaly detection problems, it is commonly assumed that the electric conductivity is constant apart from the anomaly. Without loss of generality let us assume that $\kappa\equiv 1$ in $D\backslash\overline{T}$. Moreover, we assume that $T$ is simply connected and has a smooth boundary $\partial T$. In the setting which we are interested in here, $T$ models a perfect conductor. Then potential differences in $\overline{T}$ equalize instantaneously and the governing conductivity equation is the Laplace equation
\begin{equation}\label{eqn:con1}
\Delta u=0\quad\text{in }D\backslash\overline{T}
\end{equation}
with Dirichlet boundary condition on $\partial T$
\begin{equation}\label{eqn:dir}
u|_{\partial T}=c,
\end{equation}
where the constant $c$ is implicitly defined through the imposed electrode voltages. 
We consider here discrete voltage-to-current measurements performed using $N$ electrodes $E_1,..., E_N,$ attached to $\partial D$. The electrodes are modeled by disjoint surface patches which are assumed to be simply connected subsets of $\partial D$, each having a smooth boundary curve. Within the CEM, given $N$ electrodes, the electric potential $u$ satisfies the Robin boundary condition 
\begin{equation}\label{eqn:cem}
\begin{aligned}
&z{\nu}\cdot\nabla u|_{\partial D}+fu|_{\partial D}=g\quad\text{on}\ \partial D, 
\end{aligned}
\end{equation}
where the functions $f,g:\partial D\rightarrow\mathbb{R}$ are given by
\begin{equation*}
f({x}):=\sum_{l=1}^N\chi_{l}({x}),\quad g({x}):=\sum_{l=1}^N U_l\chi_{l}({x}).
\end{equation*}
Here, $\chi_l(\cdot)$ denotes the indicator function of the $l$-th electrode $E_{l}$ and the vector ${U}=(U_1,...,U_N)^T$ denotes a prescribed electrode voltage pattern. Notice that the CEM accounts for two important physical phenomena: First, the shunting effect of the highly conducting electrodes and second
the fact that the current densities are limited by the contact impedance $z:\partial D\rightarrow\R_+$, which is caused by a thin, highly resistive layer at the electrode-skin interface.
We always assume that the ground voltage has been chosen such that 
\begin{equation}\label{eqn:ground}
\sum_{l=1}^N U_l=0.
\end{equation}
For a given voltage pattern $U\in\mathbb{R}^N$ satisfying (\ref{eqn:ground}), the equations (\ref{eqn:con1}),  (\ref{eqn:dir}) and (\ref{eqn:cem}) uniquely define the potential-current pair $(u,{J})\in H^1(D)\oplus \mathbb{R}^N$ with electrode currents
\begin{equation*}
J_{l}=\frac{1}{\lvert E_l\rvert}\int_{E_{l}}{\nu}\cdot\nabla u|_{\partial D}\mathrm{d}\sigma(x),\quad l=1,...,N,
\end{equation*} 
satisfying the conservation of charges condition
\begin{equation}\label{eqn:charge}
\sum_{l=1}^N J_l=0,
\end{equation}
cf. \cite{Somersalo}. For simplicity of the presentation let us assume that $\lvert E_l\rvert=\lvert E\rvert$, $l=1,...,N$, throughout this work.

Figures \ref{fig:1} and \ref{fig:2} illustrate the EIT forward problem using the CEM. Both figures are computed using synthetic measurement data simulated via a finite element discretization, cf. \cite{Kaipioetal}. The values of the contact impedances are comparable to those measured in real-life impedance tomography, cf. \cite{Cheng}. It has been shown experimentally that the CEM can predict EIT electrode measurements up to measurement precision, cf. \cite{Cheney,Somersalo}. In Figure \ref{fig:2} notice the peaks of the current density near the electrode edges caused by the shunting effect, which lead to severe difficulties in numerical approximations via deterministic methods. In fact, the regularity of the potential decreases as the contact impedance tends to zero, cf. \cite{Darde}, which is a huge drawback since in practice one typically aims for good contacts, i.e., small contact impedances. 
\begin{figure}
\centering
\begin{picture}(160,160)

\put(-65,1){\includegraphics[width = 0.7\textwidth]{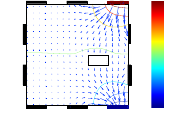}}
\put(-15,1){\small $E_9$}
\put(45,1){\small $E_8$}
\put(105,1){\small $E_7$}

\put(-15,183){\small $E_2$}
\put(45,183){\small $E_3$}
\put(105,183){\small $E_4$}

\put(134,62){\small $E_6$}
\put(134,122){\small $E_5$}

\put(-50,62){\small $E_{10}$}
\put(-50,122){\small $E_1$}

\put(184,12){\small $-1$}
\put(184,88){\small $0$}
\put(184,171){\small $1$}
\end{picture}
\caption{Current density $ \kappa\nabla u$ (arrows), equipotential lines and prescribed electrode voltages on $E_4$ and $E_7$. The box is a perfectly conducting inclusion in unit background conductivity. The forward problem of EIT is to determine the electrode currents $(J_1,...,J_{10})^T$.}\label{fig:1}
\end{figure}
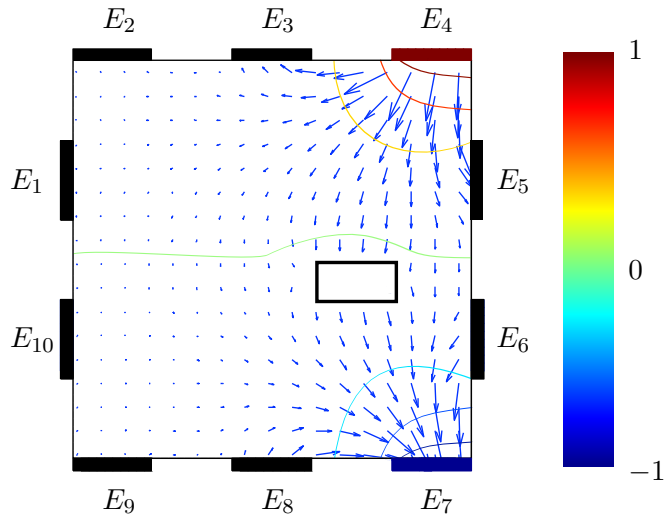

\begin{figure}
\centering
\begin{picture}(160,160)(6,0)
\put(-40,1){\includegraphics[width = 0.63\textwidth]{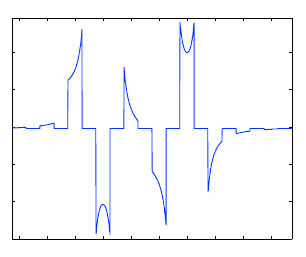}}
\put(-28,1){\small $E_1$}
\put(-6,1){\small $E_2$}
\put(17,1){\small $E_3$}
\put(40,1){\small $E_4$}
\put(62,1){\small $E_5$}
\put(85,1){\small $E_6$}
\put(107,1){\small $E_7$}
\put(130,1){\small $E_8$}
\put(152,1){\small $E_9$}
\put(173,1){\small $E_{10}$}

\put(-46,80){\rotatebox{90}{\small $ \nu\cdot\kappa\nabla u|_{\partial D}$}}
\end{picture}
\caption{Boundary current density $ \nu\cdot\kappa\nabla u|_{\partial D}$ in the CEM corresponding to the setting of Figure \ref{fig:1}. The ticks on the $x$-axis correspond to the electrode midpoints.}\label{fig:2}
\end{figure}
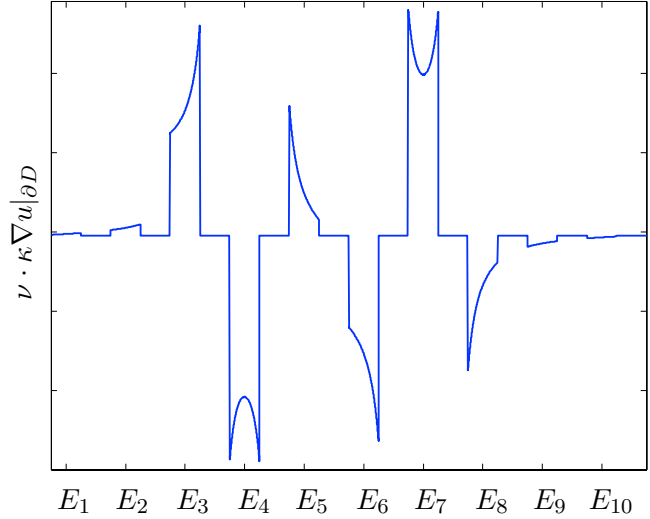

%%%%%%%%%%%%%%%%%%%%%%

\section{The standard random walk on spheres}\label{}
The \textit{random walk on spheres (RWOS) estimator} is a classical tool in stochastic numerics for elliptic and parabolic boundary value problems, originally designed to solve the Dirichlet problem for the Laplace equation
\begin{equation}
\Delta v = 0\ \text{in }D,\quad v=\phi\ \text{on }\partial D,
\end{equation}
see e.g. \cite{Muller,Mascagni0,Sabelfeld,SabelfeldTalay}.  For convenience of the reader and to introduce notations let us briefly recall the basic idea.
Let $\x_0\in \overline{D}$ and let $d_D(\x_0)$ denote the radius of the largest sphere entirely contained in $\overline{D}$ and centered in $\x_0$. 
Then the classical Feynman-Kac representation formula, cf. \cite{Revuz}, yields
\begin{equation*}
v(\x_0)=\EW[v(W_{\tau(S(\x_0,d_D(\x_0)))})|W_0=\x_0],
\end{equation*}
where $W$ is the standard $d$-dimensional Brownian motion and $\tau(S(\x_0,d_D(\x_0)))$ its first exit time from the sphere $S(\x_0,d_D(\x_0))$, cf. \cite{Karatzas}. As this representation is valid for all points inside the sphere one may use the strong Markov property of the Brownian motion to obtain the conditional expectation
\begin{equation*}
v(\x_0)=\EW[v(W_{\tau(S(\x_1,d(\x_1)))})|W_0=\x_0,\ W_{\tau(S(\x_0,d_D(\x_0)))}=\x_1].
\end{equation*}
Due to the isotropy of the Brownian motion the points $W_{\tau(S(\x,d_D(\x)))}$ are uniformly distributed over the sphere $S(\x,d_D(\x))$, i.e. the above procedure defines a time-homogenious Markov chain $\{\mathbf{X}_j\}_{j\in\mathbb{N}_0}$ with state space $(\overline{D},\B(\overline{D}))$ and initial distribution given by the Dirac measure $\delta_{\x_0}$ concentrated at $\x_0$. The states of this chain can be computed via the relation
$$\mathbf{{X}}_j=\mathbf{{X}}_{j-1}+R_j d_D(\mathbf{{X}}_{j-1}),\quad j\geq 1,$$
where $\{R_j\}_{j\in\mathbb{N}}$ is a sequence of random independent and isotropic unit vectors. For a particular realization of the random vector $\mathbf{X}_j$ we will use the lower case symbol $\x_j(\omega)$, where $\omega$ is an elementary element from the probability space $(\Omega,\mathcal{F},\mathbb{P}_{\x_0})$ of the Markov chain. We will suppress the $\omega$ in our notation if this causes no confusion. It can be shown that $\lim_{j\rightarrow\infty}\mathbf{X}_j=\mathbf{X}_{\infty}\in\partial D$, $\mathbb{P}_{\x_0}$-a.s. and the value $v(\x_{\infty})=\phi(\x_{\infty})$ is given by the Dirichlet boundary condition, which yields the probabilistic estimator  
$$\hat{v}:\overline{D}\times\Omega\rightarrow \mathbb{R},\quad\hat{v}(\x_0,\omega)=v(\x_{\infty}(\omega)).$$

To obtain a \textit{practically realizable} estimator, one usually introduces an $\ep$-layer $$D_{\ep}:=\{x\in\overline{D}:d(x,\partial D)\leq\ep\},$$ where $d(\cdot,\partial D)$ denotes the Euclidean distance to the boundary. Let $S_0(\x_j,\ep)$ denote the surface of the sphere $S(\x_j,d_D(\x_j))$ that belongs to $D_{\ep}$. The probability of $\x_{j+1}$ lying in $D_{\ep}$ is then given by \begin{equation}\label{eqn:term_prob}
\frac{S_0(\x_j,\ep)\Gamma(d/2)}{2\pi^{d/2} (d_D(\x_j))^{d-1}}.
\end{equation}
In the case of the Dirichlet problem we are interested in the discrete first hitting time, i.e., the index $\tau(D_{\ep})=\inf\{j:\x_j\in D_{\ep}\}$.
By the spherical mean value theorem we have for all $l\in \mathbb{N}_0$
\begin{equation*}
\EW[v(\mathbf{X}_{l+1})|\x_1,...,\x_l]=\EW[v(\mathbf{X}_{l+1})|\x_l]=v(\x_l), 
\end{equation*}
that is, the sequence $\{v(\mathbf{X}_j)\}_{j\in\mathbb{N}_0}$ is a discrete-time martingale with respect to $\{\mathbf{X}_j\}_{j\in\mathbb{N}_0}$ and thus by Doob's optional stopping theorem, cf.\cite{Karatzas}, the stopped chain $\{v(\x_0),...,v(\mathbf{X}_{\tau(D_{\ep})})\}$ is a one as well, implying
\begin{equation*}
\EW[\hat{v}(\x_0,\cdot)|\x_1,...,\x_l,\tau(D_{\ep})\geq l]=\EW[v(\mathbf{X}_{\tau(D_{\ep})})|\x_l,\tau(D_{\ep})\geq l]=v(\x_l).
\end{equation*} 
In particular we have $\EW[\hat{v}(\x_0,\cdot)]=v(\x_0)$ and thus $\hat{v}(\x_0,\cdot)$ is an unbiased estimator for the solution of the Dirichlet problem. The corresponding practically realizable estimator is given by
$$\hat{v}^{\ep}(\x_0,\cdot)=\phi(\pi_{\partial D}(\mathbf{X}_{\tau(D_{\ep})})),$$ 
where $\pi_{\partial D}(\x_{\tau(D_{\ep})})$ denotes the normal projection on the boundary $\partial D$. 

By the law of large numbers, $v(\x_0)$ may be approximated by simulation of i.i.d. sample paths of the chain $\{\mathbf{X}_1,...,\mathbf{X}_{\tau(D_{\ep})}\}$. Moreover, one can show that the bias of the practically realizable estimator $\hat{v}^{\ep}(\x_0,\cdot)$ is of order $\mathcal{O}(\ep)$, $\ep\rightarrow 0$, i.e., for sufficiently small $\ep$ there exist a constant $C>0$ such that the root mean square error can be estimated by
\begin{equation*}
\EW\Big[\Big(\frac{1}{M}\sum_{m=1}^M\hat{v}^{\ep}(\x_0,\omega_m)-v(\x_0)\Big)^2\Big]\leq C\Big(\ep^2+\frac{\Var[\hat{v}^{\ep}(x_0,\cdot)]}{M}\Big),
\end{equation*}
cf. \cite{Mascagni0, SabelfeldTalay}.

%%%%%%%%%%%%%%%%%%%%%%

\section{The partially reflecting RWOS estimator}\label{subsec:refrwos}
Now let us turn to the derivation of the partially reflecting RWOS estimator for the electrode currents $J_l$, $l=1,...,N$. For simplicity of the presentation we restrict ourselves here to the case $d=2$, however the generalization to $d=3$ is straightforward. Throughout this section, we assume for simplicity of the analysis of the proposed estimator that the electrodes cover the whole boundary $\partial D$, more precisely, we consider the boundary condition (\ref{eqn:Robin}) with smooth function $\phi$. The case of a finite number of discrete electrodes not covering the whole boundary is treated in the subsequent section.

Let us consider the mixed Dirichlet-Robin boundary value problem arising from the anomaly detection problem (\ref{eqn:con1}), (\ref{eqn:dir}), (\ref{eqn:Robin}) and let us assume for the moment that the constant $c$ in (\ref{eqn:dir}) is known. Moreover, we assume that $\phi$ is a smooth function. Let $\ep$ be sufficiently small such that the $\ep$-layers do not intersect, i.e., $D_{\ep}\cap T_{\ep}=\emptyset$.
In order to derive a probabilistic estimator for the potential $u(\x)$ at an arbitrary point $\x\in \overline{D}\backslash T$, we must approximate the partially reflecting Brownian motion starting in $\x$ with absorption in $T$. Therefore, let us define a time-homogeneous  Markov chain $\{\mathbf{X}_j\}_{j\in\mathbb N_0}$ with state space $(\overline{D}\backslash T\cup\{\partial\},\B_{\partial}(\overline{D}\backslash T))$, where we have adjoined an isolated cemetery point $\{\partial\}$. This cemetery point captures the missing mass, thus accounting for the fact that the chain is neither purely absorbing nor purely reflecting. The \textit{lifetime} of the chain is $\zeta$, that is, $\mathbf{X}_j=\partial$ for all $j\geq\zeta$. 
By the strong Markov property of the Brownian motion we may use the standard RWOS, as long as the chain has not entered any of the $\ep$-layers. Now let $\x_K$ denote an arbitrary state of the Markov chain inside one of the $\ep$-layers $D_{\ep}$ or $T_{\ep}$, respectively.
If $\x_K\in T_{\ep}$, the chain terminates and we have $u(\x_K)=c+r_0(\x_K)$, where $r_0(\x)=\mathcal{O}(\ep)$, $\ep\rightarrow 0$, for all $\x\in T_{\ep}$.
On the other hand, if $\x_K\in D_{\ep}$, the value of $u(\pi_{\partial D}(\x_K))$ is unknown. Without loss of generality let us assume that $\nu(\pi_{\partial D}(\x_K))=e_1$, where $e_1$ is the unit vector in direction $(1,0)^T$. 
We consider a standard 5-point stencil finite difference approximation with stepsize $h>0$ of the Laplacian
\begin{equation}\label{eqn:lap}\begin{split}
\Delta^hu(\x_K+he_1)=\frac{1}{h^2}\Big(&u(\x_K)+u(\x_K+2he_1)+u(\x_K+h(e_1+e_2))\\&+u(\x_K+h(e_1-e_2))-4u(\x_K+he_1)\Big)\end{split}
\end{equation}
together with the second order one-sided finite difference approximation of the normal derivative 
\begin{equation*}  
\nabla_{\nu}^h u(\x_K)=-\frac{1}{2h}\Big(4 u(\x_K+he_1)-3u(\x_K)-u(\x_K+2he_1)\Big).
\end{equation*}
If one of the points involved lies outside of $\overline{D}$, we reduce $h$ until all the points lie inside.
Due to the boundary condition (\ref{eqn:cem}) and the fact that $u$ is harmonic in $D\backslash \overline{T}$ we obtain thus  
\begin{equation}\label{eqn:bou}
u(\x_K)+z\nabla^h_{\nu} u(\x_k)= \phi(\pi_{\partial D}(\x_K))+r_1(\x_K),
\end{equation} 
where $r_1(\x)=\mathcal{O}(h^2+\ep/h)$ for all $\x\in D_{\ep}$. 
Now we multiply equation (\ref{eqn:lap}) by $-2zh^2$ and equation (\ref{eqn:bou}) by $2h$ and sum them up which yields for the value $u(\x_K)$ the following expression:
\begin{equation}\label{eqn:approx}
u(\x_K)=\frac{zR_hu(\x_K)}{h+z}+\frac{h\phi(\pi_{\partial D}(\x_K))}{h+z}
+hr_1(\x_K),
\end{equation}
where
\begin{equation*}
R_hu(\x_K):=\frac{1}{2}\Big(u(\x_K+h(e_1+e_2))
+u(\x_K+h(e_1-e_2))\Big).
\end{equation*}
The key observation is that the expression (\ref{eqn:approx}) yields a probabilistic interpretation, namely the first term  is the expected value of a random variable that takes the values $u(\x_K+h(e_1-e_2))$ and $u(\x_K+h(e_1+e_2))$ with equal \textit{reflection probability} 
\begin{equation}
p_r(x_K)=\frac{z}{2(h+z)}
\end{equation}
and the value $0$ with \textit{absorption probability}
\begin{equation}\label{eqn:killing_prob}
p_{a}(\x_K)=\frac{h}{h+z}.\end{equation} 

We follow the approach pursued in \cite{SabelfeldTalay} and recast this observation into an inhomogeneous integral equation of the second kind:
\begin{equation}\label{eqn:Fred}
u(\x)=\int_{\overline{D}\backslash T}u(y)k(\x, \mathrm{d} y)+F(\x),\quad \x\in \overline{D}\backslash T,
\end{equation}  
with the Radon measure
\begin{equation*}
k(\x, \mathrm{d} \y):=\begin{cases}
(1-p_{a}(\x))P_{h(e_1\pm e_2)}(\x, \mathrm{d} y),&\quad \x\in D_{\ep}\\
0,&\quad \x\in T_{\ep}\\
P_{{\overline{D}\backslash T}}(\x, \mathrm{d} y),&\quad \text{else},
\end{cases}
\end{equation*}
and inhomogeneity
\begin{equation*}F(\x):=\begin{cases}
s(\x)+hr_1(\x),&\quad \x\in D_{\ep}\\
u(\pi_{\partial T}(\x))+r_0(\x),&\quad \x\in T_{\ep}\\
0,&\quad \text{else.}
\end{cases}
\end{equation*}
$\int_B P_{h(e_1\pm e_2)}(\x, \mathrm{d} y)$, $\x\in \overline{D}\backslash{T}$, $B\in\B(\overline{D}\backslash T)$ is the probability transition kernel corresponding to the random reflection in $D_{\ep}$ according to (\ref{eqn:approx}), whereas $\int_B P_{{\overline{D}\backslash T}}(\x, \mathrm{d} y)$, $\x\in \overline{D}\backslash{T}$, $B\in\B(\overline{D}\backslash T)$ denotes the probability transition kernel corresponding to the standard RWOS in $\overline{D}\backslash T$.
Finally, the \textit{score function} of the random walk estimator is given by
\begin{equation}
s(\x):=\chi_{D_{\ep}}(\x)\frac{h\phi(\pi_{\partial D}(\x))}{h+z}.
\end{equation}

In order to obtain a probabilistic estimator, we define a randomized version of the successive approximation of the partial sums of the Neumann series 
\begin{equation*}
F(\x_0)+\sum_{j=0}^{\infty} K^j\ F(\x_0),
\end{equation*}
where $K$ denotes the integral operator in (\ref{eqn:Fred}). A canonical choice for the probability transition kernel of the underlying Markov chain $\{\mathbf{X}_j\}_{j\in\mathbb{N}_0}$ is obviously given by
\begin{equation}\label{eqn:kern}
P(\x,B):=\int_{B}k(\x,\mathrm{d} y),\quad \x\in\overline{D}\backslash T\cup \{\partial\},\ B\in\B_{\partial}(\overline{D}\backslash T).
\end{equation}
Let us denote the probability space of the Markov chain with transition kernel $P$ given by (\ref{eqn:kern}) and initial distribution $\mathbf{X}_0\sim\delta_{\x_0}$,  $\x_0\in \overline{D}\backslash T$ by $(\Omega,\mathcal{F},\mathbb{P}_{\x_0})$ and let $\EW_{\x_0}[\cdot]$ denote the expectation with respect to the measure $\mathbb{P}_{\x_0}$.
Then we may define the following \emph{partially reflecting RWOS estimator} for the electrical potential
\begin{equation}\label{eqn:collision}
\check{u}:\overline{D}\backslash T\times\Omega\rightarrow\mathbb{R},\quad\check{u}(\x_0,\omega):=\sum_{j=0}^\infty F(\x_j(\omega))
\end{equation}
as well as the corresponding practically realizable estimator
\begin{equation}\label{eqn:collision1}
\check{u}^{\ep}:\overline{D}\backslash T\times\Omega\rightarrow\mathbb{R},\quad\check{u}^{\ep}(\x_0,\omega):=\sum_{j=0}^\infty s(\x_j(\omega))+c\chi_{T_{\ep}}(\x_{\zeta-1}(\omega)).
\end{equation}
By the following result, both estimators (\ref{eqn:collision}) and (\ref{eqn:collision1}) are convergent and have uniformly bounded variance.
\begin{theorem}\label{thm:conv}
Let the boundary condition (\ref{eqn:Robin}) hold with a smooth function $\phi$. Let the Markov chain $\{\mathbf{X}_j\}_{j\in\mathbb N_0}$ have the transition kernel $P$ given by (\ref{eqn:kern}) and initial distribution $\mathbf{X}_0\sim\delta_{\x_0}$,  $\x_0\in  \overline{D}\backslash T$, then the estimator (\ref{eqn:collision}) is convergent and unbiased. Moreover, there exits a constant $C>0$, independent of $\x_0$, such that 
\begin{equation*}
\mathrm{Var}[\check{u}(\x_0,\cdot)]\leq
C\lvert\lvert F\rvert\rvert_{L^{\infty}(\overline{D}\backslash T)}^2. 
\end{equation*}
\end{theorem}
\begin{proof}
To show that the sequence of successive approximations converges uniformly it is sufficient to prove the existence of a positive constant $C_1<1$ such that $\lvert\lvert K^2\rvert\rvert_{L^{\infty}(\overline{D}\backslash T)}\leq C_1$.
For $\x\in D_{\ep}$ we have
\begin{equation*}
\int_{\overline{D}\backslash T}\int_{\overline{D}\backslash T}k(\x, \mathrm{d} y)k(y, \mathrm{d}z)\leq \max_{\x\in \partial D}\{1-p_a(\x)\}<1
\end{equation*}
and for $\x\in \overline{D}\backslash({T\cup D_{\ep}\cup T_{\ep}})$ we may split the integral to obtain
\begin{equation*}
\begin{split}
&\int_{\overline{D}\backslash T}\int_{\overline{D}\backslash({T\cup D_{\ep}\cup T_{\ep}})}P_{{\overline{D}\backslash T}}(\x, \mathrm{d}y)k(\y, \mathrm{d}z)+\int_{\overline{D}\backslash T}\int_{D_{\ep}}P_{\overline{D}\backslash T}(\x,\mathrm{d} y)k(\y,\mathrm{d}z)\\
&\leq  1- \frac{S_0(\x,\ep)}{2\pi d_{\overline{D}\backslash T}(\x)}+\max_{\x\in \partial D}\{1-p_a(\x)\}\frac{S_0(\x,\ep)}{2\pi d_{\overline{D}\backslash T}(\x)}<1,
\end{split}
\end{equation*}
where we have used formula (\ref{eqn:term_prob}) specialized to the two-dimensional case and domain $\overline{D}\backslash T$. The convergence of the sequence of successive approximations yields convergence and unbiasedness of the estimator (\ref{eqn:collision}).\\ 
Now let us write the estimator $\check{u}(\x_0,\omega)$ as the sum of local scores given by $s_{j}(\omega):=F(\x_j(\omega))$ for all $j<\zeta(\omega)$ and $s_{j}(\omega):=0$ for all $j\geq\zeta(\omega)$. Obviously it holds that
\begin{equation*}
\Var[\check{u}(\x_0,\cdot)]\leq\EW_{\x_0}\Big[\Big(\sum_{j=0}^{\infty}s_{j}(\cdot)\Big)^2\Big].
\end{equation*}
By convergence of the successive approximations, the lifetime $\zeta$ is $\mathbb{P}_{\x_0}$-a.s. finite, implying that the series of local scores is absolutely convergent in square mean with respect to the probability space $(\Omega,\mathcal{F},\mathbb{P}_{\x_0})$. In particular we may write
\begin{equation*}
\EW_{\x_0}\Big[\Big(\sum_{j=0}^{\infty}s_{j}(\cdot)\Big)^2\Big]=2\sum_{j=0}^{\infty}\sum_{k=j}^{\infty}\EW_{\x_0}[s_{j}(\cdot)s_{k}(\cdot)]-\sum_{j=0}^{\infty}\EW_{\x_0}[s_{j}(\cdot)^2].
\end{equation*}
By conditioning we obtain for $j\leq k$ 
\begin{equation*}
\EW_{\x_0}[s_j(\cdot)s_k(\cdot)]=\EW_{x_0}[s_j(\cdot)s_k(\cdot)|j<\zeta]\cdot\mathbb{P}_{\x_0}(j<\zeta)
\end{equation*}
and a straightforward calculation gives 
\begin{equation*}\begin{split}
\EW_{\x_0}[s_{j}(\cdot)s_{k}(\cdot)]&=\int_{(\overline{D}\backslash T)^{k}}F(\x_j)F(\x_k)k(\x_0,\mathrm{d} \x_1)...k(\x_{k-1},\mathrm{d} \x_k)\\
&=K^j(FK^{k-j}F)(\x_0).\end{split}
\end{equation*}
Summation of these expectations yields 
\begin{equation*}
\sum_{j=0}^{\infty}\sum_{k=j}^{\infty}K^j({F}K^{k-j}{F})=\sum_{j=0}^{\infty}K^j\Big({F}\sum_{l=0}^{\infty}K_t^l{F}\Big)=(\mathcal{I}-K)^{-1}({F}(\mathcal{I}-K)^{-1}{F})
\end{equation*}
and 
\begin{equation*}
\sum_{j=0}^{\infty}K^j{F}^2=(\mathcal{I}-K)^{-1}{F}^2.
\end{equation*}
Finally we arrive at
\begin{equation*}\begin{split}
\mathrm{Var}[\check{u}(\x_0,\cdot)]&\leq(\mathcal{I}-K)^{-1}(2{F}(\mathcal{I}-K)^{-1}{F}-{F}^2)(\x_0)\leq C  \lvert\lvert {F}\rvert\rvert_{L^{\infty}(\overline{D}\backslash T)}^2.
\end{split}
\end{equation*}
Notice that we have used the fact that by convergence of the successive approximation we may manipulate the Neumann series to obtain 
$(\mathcal{I}-K)^{-1}=(\mathcal{I}-K^2)^{-1}(\mathcal{I}+K)$ implying
\begin{equation*}
\lvert\lvert (\mathcal{I}-K)^{-1}\rvert\rvert_{L^{\infty}(\overline{D}\backslash T)}\leq C_2= 2(1-C_1)^{-1}.
\end{equation*}
We have thus shown the assertion with $C=C_2(2 C_2+1)$.
\end{proof}

Let us conclude this section with an estimate for the mean square error of the partially reflecting RWOS estimator.
\begin{theorem}\label{thm:ms}
Let the boundary condition (\ref{eqn:Robin}) hold with a smooth function $\phi$. For given $\ep>0$ there exist a stepsize $h>0$ and a constant $C>0$, such that
\begin{equation*}
\EW\Big[\Big(\frac{1}{M}\sum_{m=1}^M\check{u}^{\ep}(x_0,\omega_m)-u(x_0)\Big)^2\Big]\leq C\Big(h^4+\frac{\mathrm{Var}[\check{u}^{\ep}(x_0,\cdot)]}{M}\Big).
\end{equation*}
\begin{proof}
The mean square error is equal to
\begin{equation*}
\begin{split}
&\EW\Big[\Big(\sum_{m=1}^M\check{u}^{\ep}(x_0,\omega_m)-\EW[\sum_{m=1}^M\check{u}^{\ep}(x_0,\omega_m)]\Big)^2\Big]+\Big(\EW[\sum_{m=1}^M\check{u}^{\ep}(x_0,\omega_m)]-u(x_0)\Big)^2\\
&=\frac{\Var[\check{u}^{\ep}(x_0,\cdot)]}{M}+\Big(\EW[\sum_{m=1}^M\check{u}^{\ep}(x_0,\omega_m)]-u(x_0)\Big)^2,
\end{split}
\end{equation*}
where the first term on the right-hand side is due to the Monte Carlo sampling error and the second term is due to the bias of the discretization. As the RWOS simulates the exit position exactly, the bias only comes from the finite difference discretization when hitting the boundary. 
It suffices to consider the case $T=\emptyset$, as the variance of the estimator achieves its maximum in this case. Note that due to (\ref{eqn:bou}), $\ep\approx h^3$ is required to achieve a local bias $\mathcal{O}(h^3)$. 
When the boundary is hit, the trajectory is absorbed with probability $\frac{h}{h+z}$ and the number of hits of the boundary follows a geometric distribution with this probability as parameter. The mean number of hits is thus given by $1+\frac{z}{h}$. Consequently the global bias is of order $\mathcal{O}(zh^2)$.
\end{proof}
\end{theorem}

%%%%%%%%%%%%%%%%%%%%%%

\section{Approximation of electrode measurements}
As we have obtained a convergent partially reflecting RWOS estimator with uniformly bounded variance for the potential, we can immediately define an estimator for the electrode currents. 
Indeed we may write the potential as a sum $u=u_0+cu_1$, where $u_0$ and $u_1$ solve auxiliary boundary value problems for the Laplace equation (\ref{eqn:con1}) subject to the boundary conditions 
$$u_0=0 \text{ on }\partial T,\quad z\nu\cdot\nabla u_0|_{\partial D}+fu_0|_{\partial D}=g\text{ on }\partial D,$$ respectively, $$u_1=1\text{ on }\partial T,\quad z\nu\cdot\nabla u_1|_{\partial D}+fu_1|_{\partial D}=0\text{ on }\partial D.$$ From the boundary condition (\ref{eqn:cem}) one obtains 
\begin{equation}\label{eqn:current}
J_l=\frac{1}{\lvert E\rvert}\int_{E_l}\frac{U_l}{z}\dd\sigma(\x)-\frac{1}{\lvert E \rvert}\int_{E_l}\frac{u_0(\x)+cu_1(\x)}{z}\dd\sigma(\x),\quad l=1,...,N,
\end{equation}
and the conservation of charges condition (\ref{eqn:charge}) yields
\begin{equation*}
c = \Big(\sum_{l=1}^N\int_{E_l}\frac{U_l-u_0(\x)}{z}\dd\sigma(\x)\Big)\Big({\sum_{l=1}^N\int_{E_l}\frac{u_1(\x)}{z}\dd\sigma(\x)}\Big)^{-1}.
\end{equation*}
Therefore we define for the integrals $\int_{E_l}\frac{u_i(\x)}{z(\x)}\dd\sigma(\x)$, $i=0,1$ and $l=1,...,N$, the estimators
\begin{equation}\label{eqn:intest}
\check{\xi}_{l,i}^{\ep}(M_1,M_2,\lambda,\omega):=\frac{1}{M_2}\sum_{m_2=1}^{M_2}\frac{1}{M_1}\sum_{m_1=1}^{M_1}\check{\eta}^{\ep}_i(\x_0(\lambda_{m_2}),\omega_{m_1}),
\end{equation}
using the so-called \emph{double randomization principle}, cf. \cite{Sabelfeld}. That is, the potential is computed via the partially reflecting RWOS estimator and the boundary integrals in (\ref{eqn:current}) are approximated via Monte Carlo sampling as well, using a uniform initial distribution $\mathbf{X}_0\sim\mathcal{U}(E_l),\ l=1,...,N$. Hence $\lambda$ denotes an elementary element from the corresponding probability space and
\begin{equation*}
\check{\eta}^{\ep}_i(\x_0,\omega): = \frac{\lvert E\rvert}{z}\check{u}_i^{\ep}(\x_0,\omega),\quad i=0,1.
\end{equation*}
Obviously, to estimate the expectation, it would be sufficient to construct only one Markov chain for each realization of $\lambda$. In practice, however, a splitting technique is usually used, where $M_2$ realizations of $\lambda$ and then for each of these realizations $M_1$ independent Markov chains are constructed. For the optimal choice of $M_1$ and $M_2$ we refer the reader to the book \cite{Mikhailov}.
 Convergence of the estimator (\ref{eqn:intest}) is an immediate consequence of Theorem \ref{thm:conv} which yields the following estimator for the electrode currents:  
\begin{equation}\label{eqn:FD2}
\check{J}^{\ep}:=\Big(\check{J}^{\ep}_1(M_1,M_2,\lambda,\omega),...,\check{J}^{\ep}_N(M_1,M_2,\lambda,\omega)\Big)^T,
\end{equation}
where each component $\check{J}^{\ep}_l(M_1,M_2,\lambda,\omega)$ is given by
\begin{equation*}
\frac{1}{\lvert E\rvert}\int_{E_l}\frac{U_l}{z}\dd\sigma(\x)-\frac{1}{\lvert E\rvert}\Big(\check{\xi}_{l,0}^{\ep}(M_1,M_1,\lambda,\omega)+\check{\xi}_{l,1}^{\ep}(M_1,M_1,\lambda,\omega)\cdot \check{c}^{\ep}\Big).
\end{equation*}
The constant $c$ is approximated by the combined random estimator
\begin{equation*}
\check{c}^{\ep}:=\frac{\sum_{l=1}^N\Big(\int_{E_l}\frac{U_l}{z}\dd\sigma(\x)-\check{\xi}_{l,0}^{\ep}(M_1,M_2,\lambda,\omega)\Big)}{\sum_{l=1}^N\check{\xi}_{l,1}^{\ep}(M_1,M_2,\lambda,\omega)}.
\end{equation*}
\begin{remark}\label{rem:1}
Note that in contrast to the idealized boundary condition (\ref{eqn:Robin}), the right-hand side of (\ref{eqn:cem}) presents some discontinuouities, so that the potential is merely H\"older-continuous, cf. \cite{Piiroinen,Simon1}. In particular, the fourth order convergence with respect to $h$ obtained in Theorem \ref{thm:ms} for the idealized boundary condition (\ref{eqn:Robin}) will be reduced in the case of discrete electrode measurements.
\end{remark}

%%%%%%%%%%%%%%%%%%%%%%

\section{Variance reduction}\label{sec:vr}
In order to reduce the variance of the partially reflecting RWOS estimator, we propose a combined control variates conditional sampling technique.
The basic idea of the control variates technique is to employ the known solution of a \lq nearby\rq\ problem, see e.g. \cite{GobetMaire}. We shall exploit here the continuous dependence of the electrode currents on the conductivity.
Let us thus consider the forward problem for the homogeneous medium with unit conductivity $\kappa\equiv 1$ in $\overline{D}$ and let $v$ denote the corresponding solution of the Laplace equation in $D$, subject to the boundary condition (\ref{eqn:cem}). We may proceed as above and consider the inhomogenious integral equation
\begin{equation}
v(\x)=\int_{D}v(\y)\tilde{k}(\x,\mathrm{d} \y)+\tilde{F}(\x),\quad \x\in \overline{D},
\end{equation}  
with the Radon measure
\begin{equation}
\tilde{k}(\x, \mathrm{d} \y):=\begin{cases}
(1-p_a(\x))P_{h(e_1\pm e_2)}(\x, \mathrm{d} \y),&\quad \x\in D_{\ep}\\
P_{\overline{D}}(\x,\mathrm{d} \y),&\quad \x\in T_{\ep}\cup T\\
P_{\overline{D}\backslash T}(\x, \mathrm{d} \y),&\quad \text{else},
\end{cases}
\end{equation}
and inhomogeneity 
\begin{equation*}\tilde{F}(\x):=\begin{cases}
s(\x)+hr_0(\x),&\quad \x\in D_{\ep}\\
0,&\quad \text{else}.
\end{cases}
\end{equation*}
As in the derivation of the partially reflecting RWOS estimator we define a Markov chain $\{\tilde{\mathbf{X}}_j\}_{j\in\mathbb{N}_0}$, this time, however, with state space $(\overline{D}\cup\{\partial\},\B_{\partial}(\overline{D}))$ rather than $(\overline{D}\backslash T\cup\{\partial\},\B_{\partial}(\overline{D}\backslash T))$ and with transition kernel 
\begin{equation}\label{eqn:kern1}
\tilde{P}(\x,B):=\int_{B}\tilde{k}(\x,\mathrm{d} \y),\quad \x\in\overline{D}\cup \{\partial\},\ B\in\B_{\partial}(\overline{D}),
\end{equation}
and initial distribution $\tilde{\mathbf{X}}_0\sim \delta_{\x_0}$. 
The key idea is that we may use one realization of this Markov chain to compute a realization of both, 
$\check{v}^{\ep}(\x_0,\cdot)$, as well as $\check{u}^{\ep}_i(\x_0,\cdot)$, $i=0,1$. Let us define the practically realizable estimators
$$\tilde{\eta}^{\ep}_i(\x_0,\cdot):=\frac{\lvert E\rvert}{z}\Big(v(\x_0)-\check{v}^{\ep}(\x_0,\cdot)+\check{u}_i^{\ep}(\x_0,\cdot)\Big),\quad i=0,1.$$
Let $\tilde{\zeta}$ denote the lifetime of  $\{\tilde{\mathbf{X}}_j\}_{j\in\mathbb{N}_0}$, then we obtain for $i=0,1,$ the conditional expectations
\begin{equation}\label{eqn:exp}
\EW_{\x_0}[\tilde{\eta}^{\ep}_i(\x_0,\cdot)|\tau(T_{\ep})]=\frac{\lvert E\rvert}{z}\begin{cases}
(v(\x_0)-v(\x_{\tau(T_{\ep})})+\mathcal{O}(\ep)),\quad&\tau(T_{\ep})<\tilde{\zeta}\\
v(\x_0),\quad&\tau(T_{\ep})>\tilde{\zeta}
\end{cases}.
\end{equation}
Now we set for $i=0,1,$
\begin{equation*}
\hat{\eta}^{\ep}_i(\x_0,\cdot):=\EW_{\x_0}[\tilde{\eta}^{\ep}_i(\x_0,\cdot)|\tau(T_{\ep})],
\end{equation*}
and taking the expectation yields for $l=1,...,N$ and $i=0,1$
\begin{equation*}
\EW_{\mathcal{U}(E_l)}[\EW_{\x_0(\cdot)}[\hat{\eta}^{\ep}_i(\x_0(\cdot),\cdot)]]=\int_{E_l}\frac{u_i(\x_0)}{z}\dd\sigma(\x_0)+\mathcal{O}(\ep+h^2).
\end{equation*}
In particular the variance of $\hat{\eta}^{\ep}_i(\x_0,\cdot)$ is strictly smaller than the variance of $\tilde{\eta}^{\ep}_i(\x_0,\cdot)$ since 
\begin{equation*}
\Var[\tilde{\eta}^{\ep}_i(\x_0,\cdot)]=\Var[\hat{\eta}^{\ep}_i(\x_0,\cdot)]+\EW_{\x_0}[\Var[\tilde{\eta}^{\ep}_i(\x_0,\cdot)|\tau(T_{\ep})]].
\end{equation*}
On top of that, we use $v$ as a control variate. That is, either $v$ is known explicitly, which is the case for certain geometries, cf. \cite{Demidenko,Pidcock}, or an approximation of $v$ via a finite element or boundary element method is computed in a pre-computation step. 
In both cases we only need to simulate realizations of the random variable $\tilde{\mathbf{X}}_{\tau(T_{\ep})\wedge \tilde{\zeta}}$ and then evaluate (\ref{eqn:exp}).
In order to approximate the electrode currents we proceed as above and define for the integrals $\int_{E_l}\frac{u_i(\x)}{z}\dd\sigma(\x)$, $l=1,...,N$, $i=0,1$, the estimators
\begin{equation*}
\hat{\xi}_{l,i}^{\ep}(M_1,M_2,\lambda,\omega):=\frac{1}{M_2}\sum_{m_2=1}^{M_2}\frac{1}{M_1}\sum_{m_1=1}^{M_1}\hat{\eta}^{\ep}_i(\x_0(\lambda_{m_2}),\omega_{m_1}),\quad i=0,1,
\end{equation*}
which yields a reduced variance estimator $\hat{J}^{\ep}$ for the electrode currents if we substitute $\check{\xi}^{\ep}_{l,i}$ with $\hat{\xi}^{\ep}_{l,i}$ in the equations defining (\ref{eqn:FD2}). 

%%%%%%%%%%%%%%%%%%%%%%

\section{Generalizations}\label{sec:gen}
\subsection{Layered conductivities}
More realistic models in breast cancer modeling use layered conductivity models, see, e.g., \cite{Isaacson}. Then the forward problem (\ref{eqn:con}), (\ref{eqn:dir}) and (\ref{eqn:cem}) is a diffraction problem. Simulation of diffusion processes in piecewise constant media has been studied in recent time, see, e.g. \cite{LejayMaire,Lejay,Lejay1,LejayMaire1}. The approach we adapt here was first introduced by Lejay and the first author in \cite{LejayMaire}, therefore we content ourselves here with a brief description. For the sake of simplicity we assume that $D$ is divided in two subdomains $D_1$ and $D_2:= D\backslash D_1$ such that $D_1\subset D$ and $T\subset D_1$. The interface $\Sigma:=\partial D_1$ is assumed to be smooth and the conductivities in $D_1$ and $D_2$ will be denoted $\kappa_1$ and $\kappa_2$, respectively. We proceed similarly to the derivation of the partially reflecting RWOS estimator, i.e., we use a finite difference approximation in the interface layer $\Sigma_{\ep}$
\begin{equation}\label{eqn:key}
h^2\Delta^hu(\x_K+he_1)=-h\nabla_{\nu}^h u(\x_K)+u(\x_K)-R_h u(\x_K),
\end{equation}
where we have assumed without loss of generality that $\nu(\pi_{\Sigma}(x_K))=e_1$. 
The solution $u$ of the diffraction problem is smooth in both subdomains $D_1\backslash\overline{T}$ and $D_2$, continuous on $\Sigma$ and satisfies a transmission condition, i.e., the limit
\begin{equation*}
\lim_{h\rightarrow 0}\frac{\kappa_2(u(\pi_{\Sigma}(x_K)+he_1)-u(\pi_{\Sigma}(x_K))+\kappa_1 (u(\pi_{\Sigma}(x_K)-he_1)-u(\pi_{\Sigma}(x_K)))}{h}
\end{equation*}
vanishes. Let us introduce two parameters $h_1, h_2>0$, both of order $\mathcal{O}(h)$, such that this transmission condition may be written in the form
\begin{equation*}
\kappa_2\nabla_{\nu}^{h_2}u(\pi_{\Sigma}(x_K))=\kappa_1\nabla_{\nu}^{-h_1}u(\pi_{\Sigma}(x_K))+\mathcal{O}(\kappa_2 h_2^2+\kappa_1h_1^2).
\end{equation*}
In $D\backslash\overline{D}_1$ we obtain using the standard 5-point stencil
\begin{equation*}
\kappa_2h_2^2\Delta^{h_2}u(\x_K+h_2e_1)=\mathcal{O}(\kappa_2 h_2^3)
\end{equation*}
and in $D_1$ we have similarly
\begin{equation*}
\kappa_1h_1^2\Delta^{-h_1}u(\x_K-h_1e_1)=\mathcal{O}(\kappa_1h_1^3). 
\end{equation*}
Inserting those equations into (\ref{eqn:key}) yields
\begin{equation}\label{eqn:key1}
\kappa_2u(\x_K)-\kappa_2h_2\nabla_{\nu}^{h_2}u(\x_K)-\kappa_2 R_{h_2}u(\x_K)=\mathcal{O}(\kappa_2h_2^3)
\end{equation}
and, respectively,
\begin{equation}\label{eqn:key2}
\kappa_1u(\x_K)-\kappa_1h_1\nabla_{\nu}^{-h_1}u(\x_K)-\kappa_1 R_{-h_1}u(\x_K)=\mathcal{O}(\kappa_1h_1^3).
\end{equation}
Multiplying (\ref{eqn:key1}) by $h_1$ and (\ref{eqn:key2}) by $h_2$ and summing them up, one obtains
\begin{equation}\label{eqn:key3}
u(\x_K)=\frac{\kappa_2h_1}{\kappa_2h_1+\kappa_1h_2}R_{h_2}u(\x_K)+\frac{\kappa_1h_2}{\kappa_2h_1+\kappa_1h_2}R_{-h_1}u(\x_K)+r_3(\x_K),\end{equation}
where $r_3(\x_K)=\mathcal{O}(\kappa_2h_2^2h_1+\kappa_1h_1^2h_2+\ep)$ as $h_1, h_2, \ep\rightarrow 0$.
As in the case of the partially reflecting RWOS estimator, expression (\ref{eqn:key3}) yields a probabilistic interpretation and thus a probabilistic estimator. By the strong Markov property of the partially reflecting Brownian motion one can couple this estimator accounting for the behavior at the interface $\Sigma$ with the partially reflecting RWOS estimator. The resulting estimator may be analyzed in the same manner as described above for the partially reflecting RWOS estimator with constant background conductivity. 
\subsubsection{Choice of the parameters}
We write equation (\ref{eqn:key3}) in the generic form 
\begin{equation*}
u(\x_K)=pR_{h_2}u(\x_K)+(1-p)R_{-h_1}u(\x_K)+r_3(\x_K),
\end{equation*}
where $p\in(0,1)$. There are at least three natural choices of the parameters $h_1$ and $h_2$, namely
\begin{enumerate}
\item[(i)] $h=\kappa_2h_1=\kappa_1h_2$, then $p=1-p=\frac{1}{2}$,
\item[(ii)] $h=h_1=h_2$, then $p=\frac{\kappa_2}{\kappa_2+\kappa_1}$, $1-p=\frac{\kappa_1}{\kappa_2+\kappa_1}$,
\item[(iii)] $h_1=\frac{h}{\sqrt{\kappa_2}}$, $h_2=\frac{h}{\sqrt{\kappa_1}}$, then $p=\frac{\sqrt{\kappa_2}}{\sqrt{\kappa_1}+\sqrt{\kappa_2}}$, $1-p=\frac{\sqrt{\kappa_1}}{\sqrt{\kappa_1}+\sqrt{\kappa_2}}$.
\end{enumerate}
Notice that the first choice is related to the kinetic scheme obtained in \cite{LejayMaire1}, where the direction which was originally chosen uniformly in $(0,2\pi)$ is replaced by a discrete random variable taking only 4 values. In (ii) the probabilities to go to one side of the interface correspond to those in \cite{Mascagni}. Finally (iii) may be seen as a generalization of the one-dimensional scheme based on simulation of the skew Brownian motion in \cite{LejayMaire}.  

In our numerical examples, we will also have to deal with other types of boundary conditions
and with multiple interfaces. Consequently, we chose method (ii) for all our tests to deal more
easily with the constraints on the step $h$ in order not to cross interfaces when replacing the motion.
\subsection{Uncentered walk on spheres (UWOS)}
When the physical domain is simple, one can compute the law of the exit point of the Brownian motion
starting at any point of the domain explicitely. This has been done for instance
for rectangle domains in \cite{Deaconu} or for spherical domains
in \cite{Mascagni,MaireNguyen}. We restrict ourselves here to the case $d=2$ and for a circle of radius $R$, 
centered at point $(0,0)$, the law of the exit position of a Brownian
motion starting at point $(r\cos(\theta),r\sin(\theta))$ is given by $(R\cos(\alpha),R\sin(\alpha))$,
where 
\[
\alpha:=\theta+2\arctan\Big(\frac{R-r}{R+r}\tan(\pi U)\Big)
\]
and $U$ is a uniform random variable in $[0,1]$.

\subsection{Random parameters}
In many practical situations the electrode currents depend on some random parameter $\mu$, for instance due to random contact impedances, see e.g. \cite{Kolehmainen}. In uncertainty quantification one is usually interested in calculating the expectation and the covariance of the random current measurements with respect to the law of this parameter. Computing these quantities in an efficient manner is also crucial in the Bayesian modeling error approach, cf. \cite{KaipioSomersalo}, where the statistical properties of modeling and discretization errors are estimated beforehand and subsequently used in the numerical solution of the inverse problem. As the underlying probability spaces are usually high-dimensional, uncertainty quantification suffers from the curse of dimensionality so that for many practical applications crude Monte Carlo sampling is still the method of choice.
That is, one samples an ensemble of realizations of the random parameter and solves the deterministic boundary value problem for each realization by a deterministic method such as a finite element or boundary element method. However, the burden in terms of computation time of this procedure is likely to be prohibitive.

In the framework presented here, this difficulty can be overcome naturally by using the double randomization principle, cf. \cite{Sabelfeld}, which yields the relations
\begin{equation*}\begin{split}
&\EW^{\mu}[J_l(\cdot)]=\EW^{\mu}[\EW^{(\omega,\lambda)}[\hat{J}_l(\cdot,\cdot,\mu)|\mu]]=\EW^{\omega,\lambda,\mu}[\hat{J}_l(\cdot,\cdot,\cdot)]\\ 
&\mathrm{Cov}^{\lambda}[J_l(\cdot)J_m(\cdot)]=\EW^{(\omega_1,\omega_2,\lambda_1,\lambda_2,\mu)}[\hat{J}_l(\omega_1,\lambda_1,\mu)\hat{J}_m(\omega_2,\lambda_2,\mu)].
\end{split}
\end{equation*}
Here $\omega_1,\omega_2$ and $\lambda_1,\lambda_2$, respectively, are conditionally independent trajectories constructed for a fixed realization  of $\mu$.

%%%%%%%%%%%%%%%%%%%%%%

\section{Numerical tests}
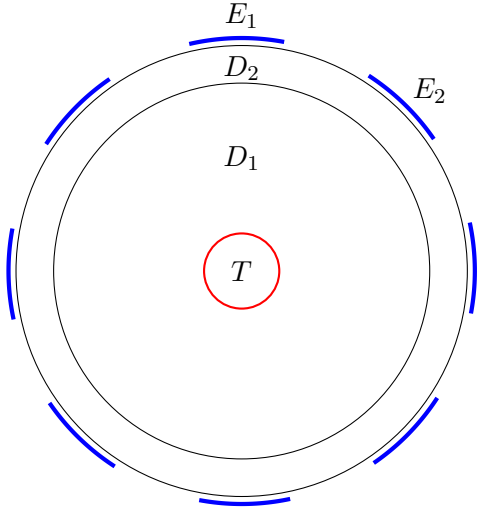
\begin{figure}\vspace{-3cm}
\center
\begin{tikzpicture}[xscale=-1,
    ray/.style={decoration={markings,mark=at position .5 with {
      \arrow[>=latex]{>}}},postaction=decorate}
  ]

  \pgfmathsetlengthmacro{\r}{3cm}
   \pgfmathsetlengthmacro{\reps}{2.5cm}
 
    \pgfmathsetlengthmacro{\rr}{0.5cm}
    \pgfmathsetlengthmacro{\rrr}{0.9cm}
\pgfmathsetlengthmacro{\rrrr}{1.55cm}

  \pgfmathsetmacro{\f}{.7}

  \coordinate (O) at (0, 0);
  \coordinate (OO) at (0, 0);
  \coordinate (OOO) at (0.8,0.84);
  \coordinate (O1) at (-0.8, 2.8);
  \coordinate (O2) at (0.3422, 1.4);

\node at (.0,2.7) {\small$D_2$};
\node at (.0,3.4) {\small$E_1$};
\node at (-2.5,2.4) {\small$E_2$};
\node at (.0,1.5) {\small$D_1$};
\node at (0,0) {\small$T$};

   \draw [blue,ultra thick,domain=0:23.5,rotate=77] plot ({3.1*cos(\x)}, {3.1*sin(\x)});
   \draw [blue,ultra thick,domain=46:68.5,rotate=77] plot ({3.1*cos(\x)}, {3.1*sin(\x)});
   \draw [blue,ultra thick,domain=91:113.5,rotate=77] plot ({3.1*cos(\x)}, {3.1*sin(\x)});
   \draw [blue,ultra thick,domain=136:158.5,rotate=77] plot ({3.1*cos(\x)}, {3.1*sin(\x)});   
   \draw [blue,ultra thick,domain=181:203.5,rotate=77] plot ({3.1*cos(\x)}, {3.1*sin(\x)});
   \draw [blue,ultra thick,domain=226:248.5,rotate=77] plot ({3.1*cos(\x)}, {3.1*sin(\x)});  
   \draw [blue,ultra thick,domain=271:293.5,rotate=77] plot ({3.1*cos(\x)}, {3.1*sin(\x)});
   \draw [blue,ultra thick,domain=316:338.5,rotate=77] plot ({3.1*cos(\x)}, {3.1*sin(\x)});

  \draw (O) circle (\r);
  \draw (O) circle (\reps);
  
  \draw[red, thick] (OO) circle (\rr);

\end{tikzpicture}\caption{The benchmark setting modeling a breast geometry with 8 electrodes and a layered conductivity. The electrodes are numbered clockwise.}\label{fig:rwo}
\end{figure}
Our numerical tests were performed using a circular model, i.e., $D$ is the planar unit circle, see Figure \ref{fig:rwo}. Such a  geometry may serve as an appropriate model for certain mammography systems, cf., e.g., \cite{Azzouz,Georgi}.
We chose a circular inclusion of radius $r$ centered at the origin. In this case the constant on the boundary of the inclusion must be equal to zero for symmetry reasons. We used $8$ electrodes, each of width $0.1$. An alternating voltage pattern was imposed, i.e., $U_j=(-1)^j$, $j=1,...,8$. 
For the computation of the electrode currents, we used the double randomization principle, that is, the starting point of each trajectory was picked uniformly at random
on one of the electrodes. 

The numerical scheme was implemented in Fortran and parallelized via OpenMP. The test cases were run on a workstation with 4 AMD Opteron 8 core CPUs. Pseudo random numbers were generated with an implementation of L'Ecuyers's parallel MRG32k3a random number generator. The reference solution was computed using finite element routines from the EIDORS package, cf. \cite{Vaukhonen}.

\subsection{Unit background conductivity}
\subsubsection{Idealized measurement model}
To verify the theoretical result of Theorem 2 numerically, let us first study the idealized measurement model assuming that measurements can be taken on the whole boundary. 
In our experiment, we considered the Robin boundary condition 
$$z\nabla u|_{\partial D}+u|_{\partial D}=\cos(4\theta),$$
where $\theta$ denotes the polar angle. In order to analyse the global bias of the partially reflecting RWOS estimator, we computed the bias $B_{z_{1}},B_{z_{2}}$ of the practically realizable estimator (\ref{eqn:collision1}) at a single point $x=(0.99361,0.11286)$ for two different values of the contact impedance $z$,
namely 
\[
z_{1}:=0.5,\quad z_{2}:=0.1
\]
and 5 different values of $h$ chosen equidistantly from the interval $[0.08,0.2].$
The reference values for a centered circular inclusion of radius $r=0.3$ were computed using 
a very fine discretization by linear finite elements, cf. \cite{Vaukhonen}; we found $u(x)\approx 0.299$ and $u(x)\approx 0.642$, respectively. 
In Figure \ref{fig:4}, we plot for each contact impedance the bias corresponding to the approximations obtained by the new estimator using $10^{6}$ simulations and the different stepsizes in a logarithmic scale together with the corresponding least-square fits. 
We obtained an estimated order of convergence (EOC) of $2.05$ and $2.07$ for $z_1$ and $z_2$, respectively. 
Moreover, as one would expect from the proof of Theorem 2, the bias increased, when $z$ increased.

\begin{figure}
\centering
\begin{picture}(160,180)(0,-26)
\put(-80,-20){\includegraphics[width = 0.83\textwidth]{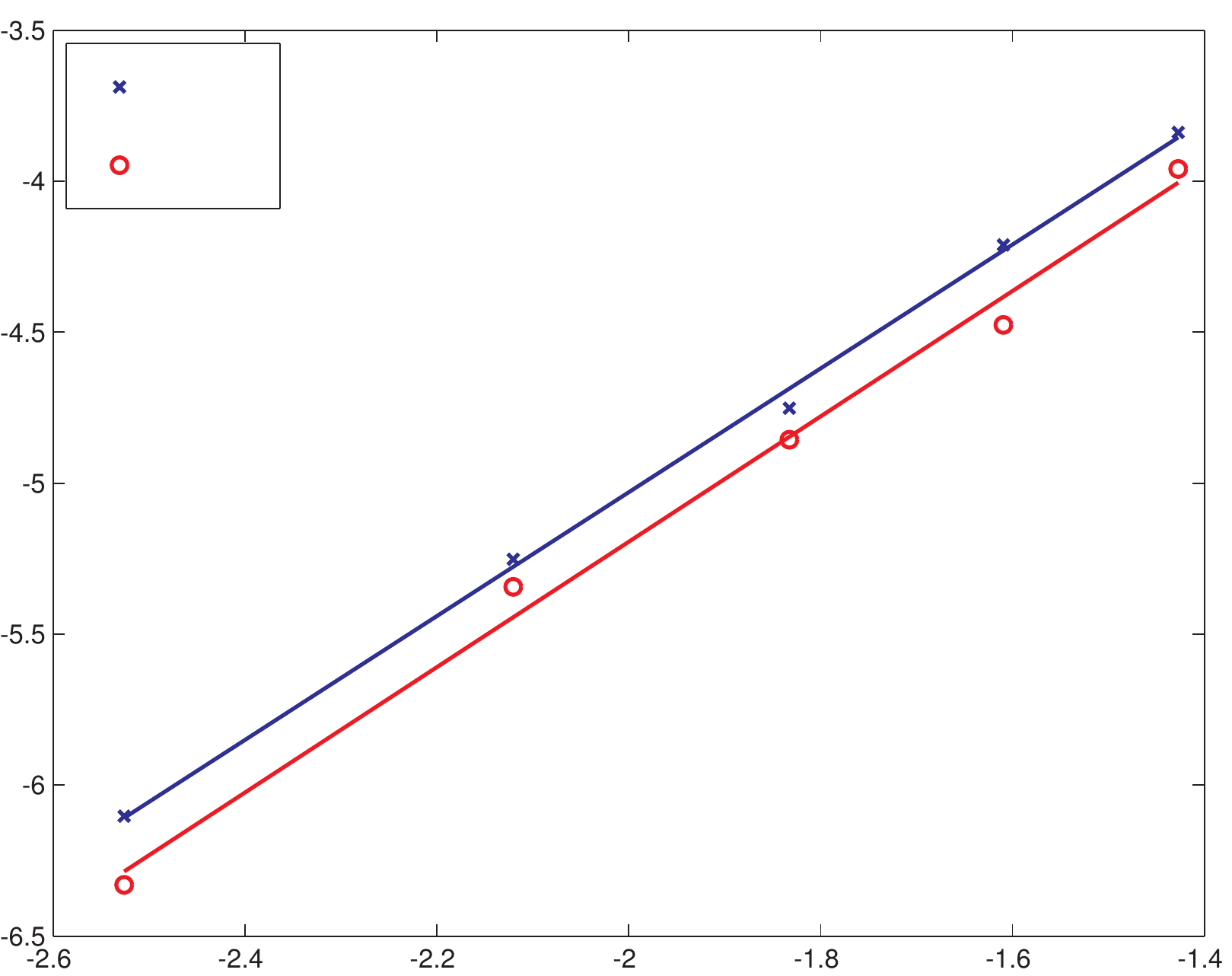}}
\put(-100,86){\rotatebox{90}{\small $\log(B_{z_i})$}}
\put(76,-32){\small $\log(h)$}
\put(-36,210){\small $z_1$}
\put(-36,190){\small $z_2$}
\end{picture}
\caption{Idealized measurement model: The logarithm of the bias is plotted against the logarithm of the stepsize; solid lines show the corresponding least-squares fits.}\label{fig:4}
\end{figure}
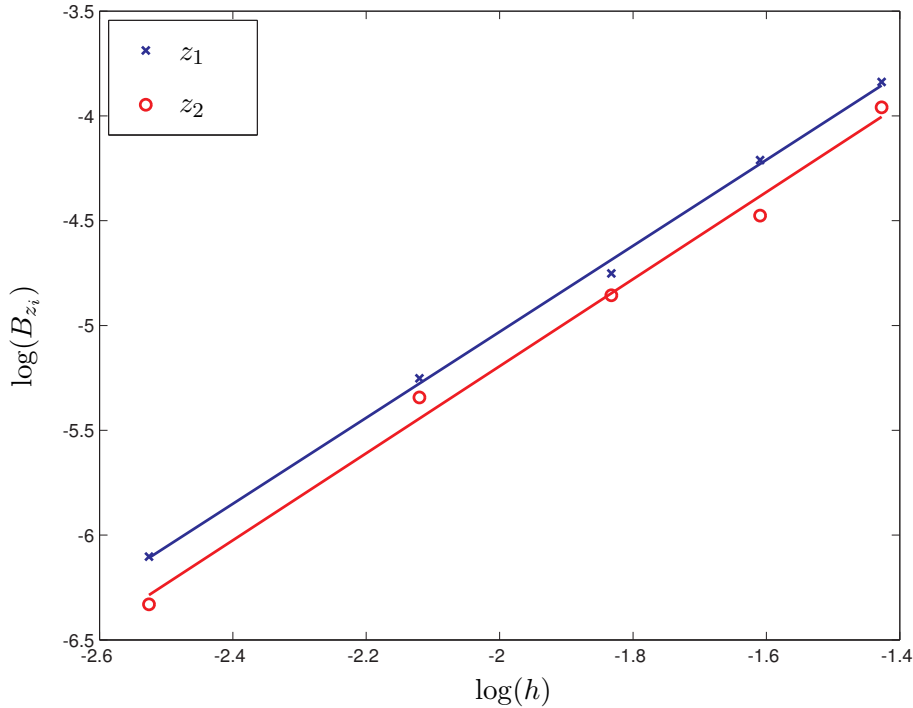

\subsubsection{Discrete electrode measurements}

\paragraph{Bias estimation}
In this experiment, we computed for the direct method (i.e. the method without variance reduction) the bias $B_{z_{i}}$, $i=1,2$,
of the electrode current through the electrode $E_3$ centered at $(1,0)$ for 5 different
values of $h$ chosen equidistantly from the interval $[0.04,0.2].$ 

Again, the reference values for a centered circular inclusion of radius $r=0.3$ were computed using 
a very fine discretization by linear finite elements, see Figure \ref{fig:fem}. In Figure \ref{fig:5}, we plot for each contact impedance the bias corresponding to the approximations obtained by the new estimator using $10^{6}$ simulations and the different stepsizes in a logarithmic scale together with the corresponding least-square fits. We obtained an EOC of $1.76$ and $1.62$ for $z_1$ and $z_2$, respectively. As one would expect from Remark \ref{rem:1}, the EOC is reduced compared to the idealized measurement model.

\begin{figure}[t]
\centering
\begin{picture}(160,150)(0,-6)
\put(-35,0){\includegraphics[width = 0.60\textwidth]{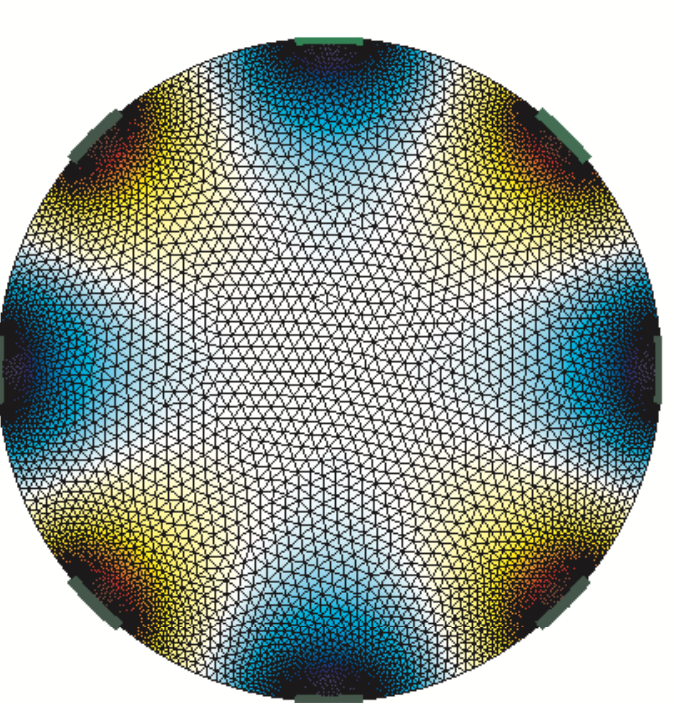}}
\end{picture}
\caption{Discrete electrode measurements: Finite element discretization and the reference solution within the domain $D$ for an alternating voltage pattern.}\label{fig:fem}
\end{figure}

\begin{figure}[t]
\centering
\begin{picture}(160,250)(0,-26)
\put(-80,-20){\includegraphics[width = 0.83\textwidth]{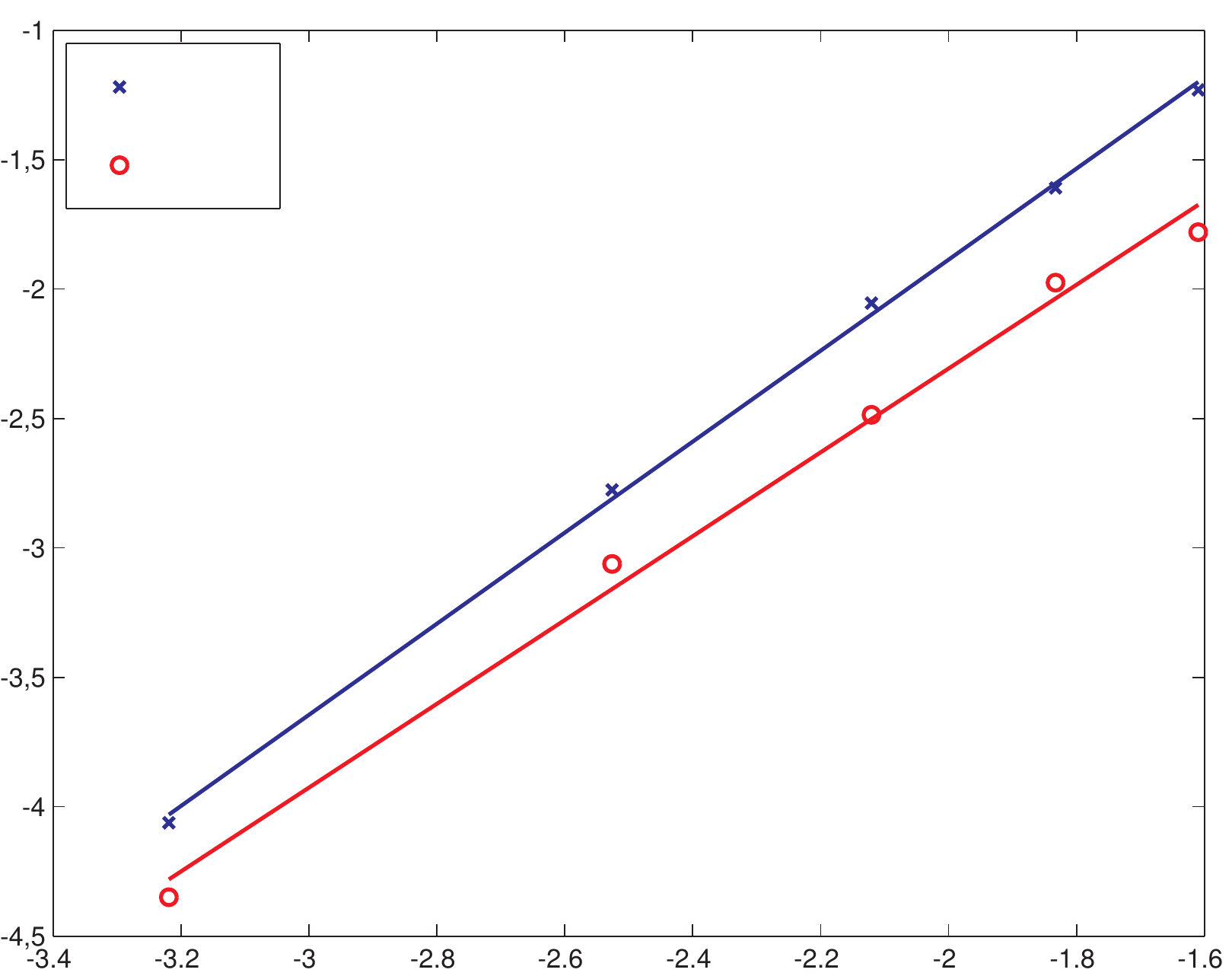}}
\put(-100,86){\rotatebox{90}{\small $\log(B_{z_i})$}}
\put(76,-32){\small $\log(h)$}
\put(-36,210){\small $z_1$}
\put(-36,190){\small $z_2$}
\end{picture}
\caption{Discrete electrode measurements: The logarithm of the bias is plotted against the logarithm of the stepsize; solid lines show the corresponding least-squares fits.}\label{fig:5}
\end{figure}

\begin{table}[b]
\centering
\begin{tabular}{|c|c|c|c|c|c|c|c|c|c|}
\hline 
$r$ & $J_3^{\text{ref}}$ & $\check{J}_3^{\ep}$ & $\check{\sigma}_3^{\ep}$ &  $\hat{J}_3^{\ep}$ & $\hat{\sigma}_3^{\ep}$ \tabularnewline
\hline 
0.9 & 0.976 & 0.974 & 0.445 & 0.974 & 0.161 \tabularnewline
\hline 
0.8 & 0.902 & 0.896 & 0.517  & 0.904 & 0.098 \tabularnewline
\hline 
0.7 & 0.874 & 0.868 & 0.612  & 0.876 & 0.054 \tabularnewline
\hline 
0.5 & 0.864 & 0.872 & 0.681  & 0.865 & 0.015 \tabularnewline
\hline 
0.3 & 0.862 & 0.871 & 0.751  & 0.862 & 0.001 \tabularnewline
\hline 
\end{tabular}\caption{Variance reduction: Approximations of the reference electrode current $J_3^{\text{ref}}$ via the direct method $\check{J}_3^{\ep}$ and via the method with variance reduction $\hat{J}_3^{\ep}$ together with the corresponding standard deviations.}\label{tab:1}
\end{table}

\paragraph{Variance reduction}
Next, we investigated the efficiency of the variance reduction method proposed in Section \ref{sec:vr}. 
The control variate, i.e., the electrode current corresponding to the problem 
without inclusion was precomputed using a finite element discretization.
Table \ref{tab:1} shows the experimental results for the test problem described in the previous paragraph with contact impedance $z_{2}$ and an inclusion centered at the origin with different radii. 

The reference values $J_3^{\text{ref}}$ for the different radii of the inclusion were computed using 
a very fine discretization by linear finite elements; 
All simulations were performed using $10^{6}$ Monte Carlo simulations with stepsize $h=0.004$ and $\ep=10^{-6}$. 
We computed approximations of the electrode current denoted 
$\check{J}_3^{\ep}$ and its standard deviation $\check{\sigma}_3^{\ep}$ for the direct method, respectively 
$\hat{J}_3^{\ep}$ and $\hat{\sigma}_3^{\ep}$ for the method with variance reduction.

As one would expect, the standard deviation
increases when $r$ decreases for the direct method. For the method with variance
reduction, the standard deviation decreases when $r$ decreases
because the problem gets \lq closer\rq\ to the one without inclusion which is used as control variate. This means in particular that the efficiency of the method with variance reduction increases as the size of the inclusion decreases which is particular interesting with regard to the inverse problem.

Subsequently, we investigated the efficiency of our approximation methods, as well as different variants thereof,
based on the quantity $C$ given by the variance multiplied by the computational
time which is a standard criterion for Monte Carlo methods. Note that for the variants of the method with variance reduction, which shall be described below, we did not include the computational time required to solve the problem without inclusion as this is a precomputation which can be done once and for all.
\begin{table}[b]
\centering
\begin{tabular}{|c|c|c|c|c|c|c|c|c|}
\hline 
$r$ & $C_{Dir}$ & $C_{FE}$ & $C_{RW}^{(1)}$ & $C_{RW}^{(2)}$ & $C_{RW}^{(10)}$ & $C_{UW}^{(1)}$ & $C_{UW}^{(2)}$ & $C_{UW}^{(10)}$\tabularnewline
\hline 
0.9 & 0.74 & 0.78 & 16.7 & 16.2 & 19.5 & 5.5 & 4.1 & 3.7\tabularnewline
\hline 
0.8 & 2.0 & 0.21 & 14.2 & 12.5 & 12.3 & 6.1 & 4.0 & 2.5\tabularnewline
\hline 
0.7 & 3.4 & 0.07 & 13.0 & 9.9 & 8.4 & 6.0 & 3.9 & 1.9\tabularnewline
\hline 
0.5 & 6.6 &3.0E-3 & 9.1 & 7.2 & 4.8 & 5.9 & 3.5 & 1.3\tabularnewline
\hline 
0.3 & 10.1 & 4.5E-5 & 7.2 & 4.8 & 2.9 & 5.5 & 2.9 & 0.9\tabularnewline
\hline 
\end{tabular}\caption{Efficiency of the methods (unit background): The direct method, the method with variance reduction based on FEM, the RWOS and the UWOS, respectively.}\label{tab:2}
\end{table}
The results for the direct method, the method with variance reduction based on FEM, the RWOS and the UWOS, respectively, are shown in Table \ref{tab:2}. 
The latter methods use the RWOS, respectively the UWOS instead of the FEM to compute the solution on the inclusion. For these methods, the superscript in the notation denotes the number of sample paths used for this computation. We see that the method with variance reduction based on FEM is more efficient than the other variants and that its efficiency increases as $r$ increases. 

\subsection{Discrete electrode measurements and layered background conductivity}
Now let us turn to the more realistic case of a layered background conductivity.
In this experiment, we considered a domain $D$ which is separated into two areas with diffusion
coefficient $\kappa_{1}$ for all points $(x,y)$ such that $\sqrt{x^{2}+y^{2}}\geq R$
and with diffusion coefficient $\kappa_{2}$ elsewhere outside the inclusion, see Figure \ref{fig:rwo}.
We chose $\kappa_{1}=1.5,\kappa_{2}=1$ and $R=0.9.$

\subsubsection{Bias estimation}
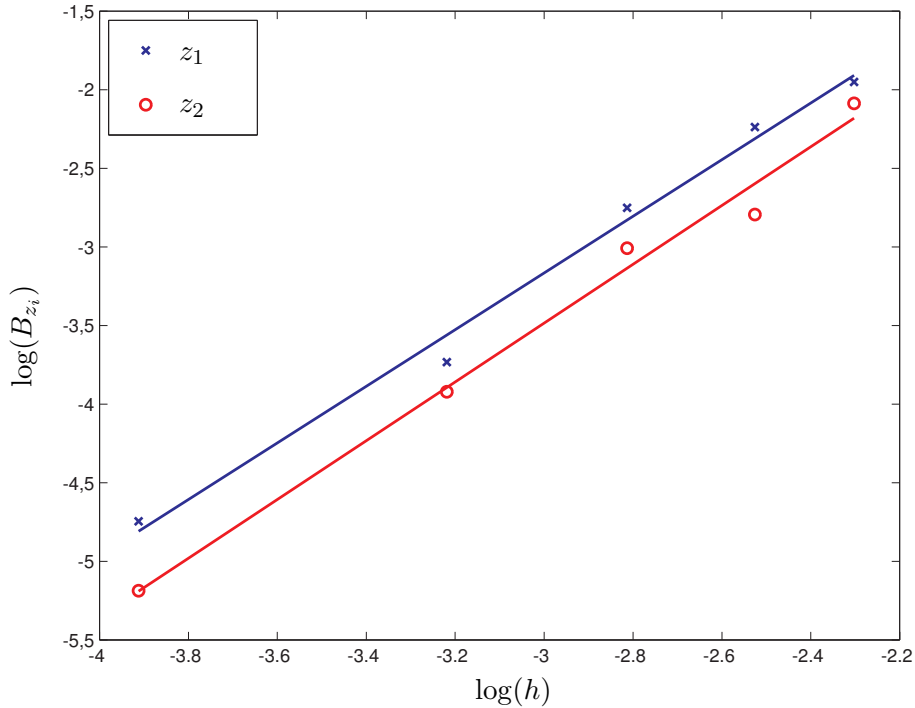
\begin{figure}
\centering
\begin{picture}(160,180)(0,-26)
\put(-80,-20){\includegraphics[width = 0.83\textwidth]{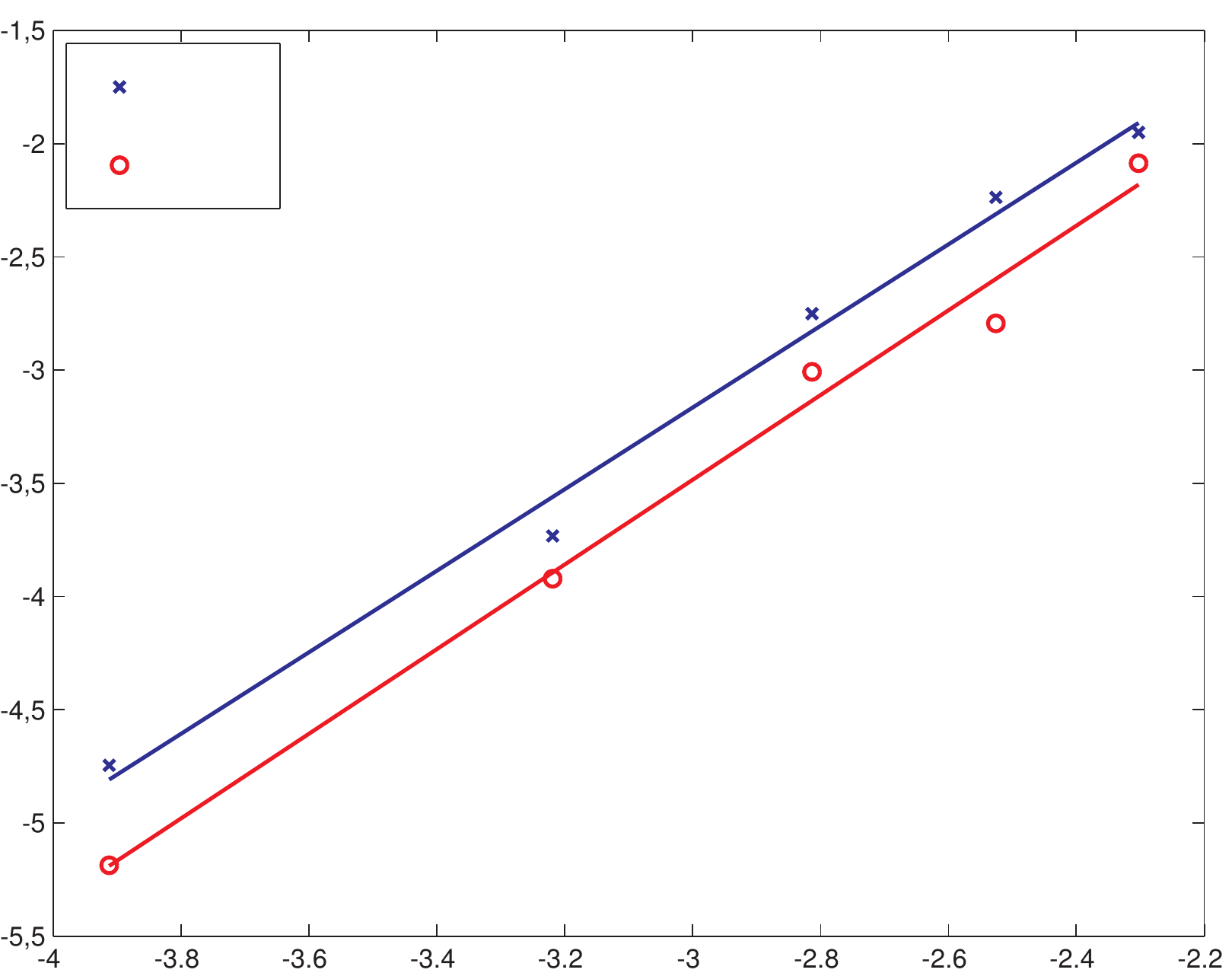}}
\put(-100,86){\rotatebox{90}{\small $\log(B_{z_i})$}}
\put(76,-32){\small $\log(h)$}
\put(-36,210){\small $z_1$}
\put(-36,190){\small $z_2$}
\end{picture}
\caption{Discrete electrode measurements and layered background conductivity: The logarithm of the bias is plotted against the logarithm of the stepsize; solid lines show the corresponding least-squares fits.}\label{fig:6}
\end{figure}
We computed the bias $B_{z_{i}}$, $i=1,2$, of the of the electrode current through the electrode $E_3$ centered at $(1,0)$ for 5 different values of $h$ chosen equidistantly in the interval $[0.02,0.1]$. As in the previous experiments, the reference values for a centered circular inclusion of radius $r=0.3$ were computed using a very fine discretization by linear finite elements. 
In Figure \ref{fig:6}, we plot the bias corresponding to the approximations obtained by the direct method using $10^{6}$ simulations in a logarithmic scale together with the corresponding least-square fits. 

We obtained an EOC of $1.80$ and $1.86$ for $z_1$ and $z_2$, respectively. 
In analogy to the observation in Remark \ref{rem:1}, the discontinuity of the conductivity leads to a merely H\"older-continuous so that the EOC is slightly smaller than two.

\subsubsection{Variance reduction}
Next, as in the previous experiment, we investigated the efficiency of the variance reduction method proposed in Section 6. The control variate, i.e., the electrode current corresponding to the problem 
without inclusion was precomputed using a finite element discretization.
As in the previous experiment, we investigated the efficiency of our approximation methods as well as different variants thereof,
based on the quantity $C$ given by the variance multiplied by the computational
time. The results for the direct method, the method with variance reduction based on FEM, the RWOS and the UWOS, respectively, are shown in Table \ref{tab:3}. 
The latter methods use the RWOS, respectively the UWOS instead of the FEM to compute the solution on the inclusion. For these methods, the superscript in the notation denotes the number of sample paths used for this computation. 
 \begin{table}
\centering
\begin{tabular}{|c|c|c|c|c|c|c|c|c|}
\hline 
$r$ & $C_{Dir}$ & $C_{FE}$ & $C_{RW}^{(1)}$ & $C_{RW}^{(2)}$ & $C_{RW}^{(10)}$ & $C_{UW}^{(1)}$ & $C_{UW}^{(2)}$ & $C_{UW}^{(10)}$\tabularnewline
\hline 
0.8 & 4 & 0.27 & 49 & 45 & 43 & 44 & 40 & 38\tabularnewline
\hline 
0.7 & 8 & 0.09 & 45 & 39 & 34 & 41 & 34 & 31\tabularnewline
\hline 
0.6 & 12 & 0.03 & 42 & 33 & 27 & 38 & 29 & 24\tabularnewline
\hline 
0.5 & 17 & 8E-3 & 31 & 25 & 20 & 30 & 23 & 17\tabularnewline
\hline 
0.3 & 28 & 9E-5 & 29 & 22 & 15 & 28 & 20 & 12\tabularnewline
\hline 
\end{tabular}\caption{Efficiency of the methods (layered background): The direct method, the method with variance reduction based on FEM, the RWOS and the UWOS, respectively.}\label{tab:3}
\end{table}
Note the increase of the values of $C$ compared to the previous method for
the direct method. This can be explained by an increase of the
computational times due to the time spent by the trajectory near the
interface. Moreover, we see that the variance reduction
method based on FE is still highly efficient and clearly superior to the direct
method.

\subsection{Random background conductivity}

In many practical situations, the background is not perfectly known and is therefore
modelled as a random medium. We consider a layered model where
the diffusion coefficient in each region is given by uniformly distributed random variable.
We assume that $\kappa_{1}$ is uniformly distributed in the interval $[1.3,1.7]$ and
$\kappa_{2}$ is uniformly distributed in the interval $[0.8,1.2]$.  We also assume
that the radius $R$ is a random variable, uniformly distributed in $[0.89,0.91].$
As described in Section \ref{sec:gen}, the computation of the mean value of the
electrode current is computed using the double randomization principle. A realization of the medium is picked
according to its distribution and a starting point for the trajectory is picked according to a uniform distribution
on the electrode; then one trajectory is simulated and the corresponding score
is computed. This procedure is repeated a number of times and the resulting scores are finally averaged.
Concerning the variance reduction, it is not obvious how to use a FEM approximation of the control variate in a straight forward way without remeshing for each realization of the random
medium. Therefore, we restrict ourselves here to the control variate based on the continuation
of the walk via the UWOS method with 10 trajectories. 
\begin{table}
\centering
\begin{tabular}{|c|c|c|c|c|c|c|c|}
\hline 
$r$ & $\mathbb{E}^{\mu}[J_3^{\text{ref}}(\cdot)]$ & $\check{J}_3^{\ep}$ & $\check{\sigma}_3^{\ep}$ & $C_{Dir}$ & $\hat{J}_3^{\ep}$ & $\hat{\sigma}_3^{\ep}$ & $C_{UW}^{(10\text{)}}$\tabularnewline
\hline 
0.7 & 0.902 & 0.896 & 0.654 & 7 & 0.902 & 0.225 & 28\tabularnewline
\hline 
0.6 & 0.874 & 0.874 & 0.703 & 12 & 0.872 & 0.207 & 22\tabularnewline
\hline 
0.5 & 0.864 & 0.860 & 0.747 & 17 & 0.865 & 0.153 & 19\tabularnewline
\hline 
0.3 & 0.862 & 0.860 & 0.830 & 28 & 0.860 & 0.102 & 12\tabularnewline
\hline 
0.2 & 0.862 & 0.860 & 0.899 & 36 & 0.863 & 0.057 & 9\tabularnewline
\hline 
\end{tabular}\caption{Variance reduction and efficiency of the methods (random background): Approximations of the reference electrode current $J_3^{\text{ref}}$ via the direct method $\check{J}_3^{\ep}$ and via the method with variance reduction $\hat{J}_3^{\ep}$ together with the corresponding standard deviations and the quantity $C$.}\label{tab:4}
\end{table}
We note that the approximation of the mean value of the solution to the
partial differential equation with a stochastic coefficient is computed without extra cost and with
a variance of the same order as the one for the deterministic problem. This is a huge advantage compared to deterministic
methods, where the solution of an elliptic boundary value problem is required
for each draw corresponding to a realization of the random medium. In our experiments, the probabilistic methods were up to 100 times faster than the Monte Carlo sampling based on FEM solution of the forward problem. Moreover, the variance reduction is still efficient to some extent which is an important aspect with regard to the inverse problem of detecting small anomalies.

\section{Summary and future work}
The complete electrode model is the most realistic model to approximate real electrode measurements in electrical impedance
tomography. We have given a probabilistic interpretation of the corresponding electrode currents taking into account both the mixed boundary condition and the possibly discontinuous diffusion coefficient.
Then, we have proposed a Monte Carlo method based an a novel partially reflecting RWOS estimator to compute these currents in an embarrassingly parallel manner. This method involves the RWOS algorithm inside subdomains where the diffusion coefficient is constant and replacement techniques motivated by finite difference discretization to deal with mixed boundary conditions as well as transmission conditions. The global bias of the corresponding algorithm is analyzed both theoretically and experimentally. Moreover the variance of the new estimator is studied and subsequently considerably reduced via a control variate conditional sampling technique. Indeed, it is this variance reduction which makes the proposed method such an interesting alternative to standard deterministic methods, even for two-dimensional problems. 

In future work, we intend to use the new Monte Carlo method in the framework of Bayesian statistical inverse problems. We expect that the inherent parallelism of our method, combined with the wide availability of multi- and many-core computing hardware, will enable the efficient treatment of three-dimensional problems which would be prohibitively expensive in terms of computation time with a (non-parallelized) deterministic forward solver.

\section*{Acknowledgements}
The research of MS was supported by the Deutsche Forschungsgemeinschaft (DFG) under grant HA 2121/8 -1 583067. MS would like to express his gratitude to Professor Martin Hanke for many inspiring discussions on the topic of this work.

\bibliographystyle{elsarticle-num}
\bibliography{<your-bib-database>}

\begin{thebibliography}{99}
\addcontentsline{toc}{chapter}{Bibliography}
\def\cprime{$'$}
\bibitem{Alessandrini}G.~Alessandrini. 
\newblock\emph{Stable determination of conductivity by boundary measurements},
\newblock Appl. Analysis, \textbf{27} (1988), 153--172.

\bibitem{Astala}
K. Astala and L. P{\"a}iv{\"a}rinta.
\newblock\emph{Calder\'on's inverse conductivity problem in the plane},
 \newblock Ann. of Math., \textbf{163} (2006), 265--299.

\bibitem{Azzouz}
\newblock M.~Azzouz, M.~Hanke, C.~Oesterlein and K.~Schilcher, 
\newblock\emph{The factorization method for electrical impedance tomography data from a new planar device} 
\newblock Int. J. Biomed. Imaging, (2007), Article ID 83016.

 \bibitem{Bossyetal}
M. Bossy, E. Gobet and D. Talay,
 \newblock\emph{A symmetrized {E}uler scheme for an efficient approximation of reflected diffusions},
 \newblock J. Appl. Probab., {41} (2004), 877--889.

\bibitem{Bossyetal1}
M. Bossy, N. Champagnat, S. Maire and D. Talay,
\newblock\emph{Probabilistic interpretation and random walk on spheres algorithms for the {P}oisson-{B}oltzmann equation in molecular dynamics},
\newblock M2AN Math. Model. Numer. Anal., {44} (2010), 997--1048.

\bibitem{Cheney}
M. Cheney, D. Isaacson and J. C. Newell,
\newblock\emph{Electrical impedance tomography},
\newblock SIAM Rev., {41} (1999), 85--10.

\bibitem{Cheng}
K.-S. Cheng, D.~Isaacson, J.~Newell and D.~G.~Gisser,
\newblock\emph{Electrode models for electric current computed tomography}, 
\newblock IEEE Transactions on Biomedical Engineering, {36} (1989), 918--924.

\bibitem{Isaacson1}
M.~H.~Choi, T.-J.~Kao, D.~Isaacson, G.~J.~Saulnier and J.~C.~Newell,
\newblock\emph{A reconstruction algorithm for breast cancer imaging with electrical impedance tomography in mammography geometry}, 
\newblock IEEE Transactions on Biomedical Imaging, {54} (2007), 700--710.

\bibitem{Constantini}
C. Costantini, B. Pacchiarotti and F. Sartoretto,
\newblock\emph{Numerical approximation for functionals of reflecting diffusion processes},
\newblock SIAM J. Appl. Math., {58}, (1998), 73--102.

\bibitem{Darde}
J. Dard{\'e}, N. Hyv\"{o}nen, A. Sepp\"{a}nen and S. Staboulis, 
\newblock\emph{Simultaneous recovery of admittivity and body shape in electrical impedance tomography: An experimental evaluation},
\newblock Inverse Problems {29} (2013), 085004.

\bibitem{Deaconu}
M.~Deaconu and A.~Lejay.
\newblock\emph{A random walk on rectangles algorithm},
\newblock{Methodol. Comput. Appl. Probab.}, \textbf{8} (2006), 135--151.

\bibitem{Demidenko}
E.~Demidenko,
\newblock\emph{An analytic solution to the homogeneous EIT problem on the 2D disk and its application to estimation of electrode contact impedances},
\newblock Physiol. Meas. {32} (2011), 1453--1471.

\bibitem{Jin}
M.~Gehre and B.~Jin,
\newblock\emph{Expectation propagation for nonlinear inverse problems -- with an application to electrical impedance tomography},
\newblock J. Comput. Phys. {259} (2014), 513--535.

\bibitem{Georgi}
K.-H.~Georgia, C.~H\"ahnlein, K.~Schilcher, C.~Sebuc and H.~Spiesberger 
\newblock\emph{Conductivity reconstructions using real data from a new planar electrical impedance tomography device}
\newblock Inverse Probl. Sci. En., {21} (2013), 801-822.

\bibitem{Gobet}
E. Gobet,
\newblock\emph{Efficient schemes for the weak approximation of reflected diffusions},
\newblock Monte Carlo Methods Appl., {7} (2001), 193--202.

\bibitem{Gobet1}
E. Gobet,
\newblock\emph{Euler schemes and half-space approximation for the simulation of diffusion in a domain},
\newblock ESAIM Probab. Statist., {5} (2001), 261--297.

\bibitem{GobetMaire}
E. Gobet and S. Maire,
\newblock\emph{Sequential control variates for functionals of {M}arkov processes},
\newblock SIAM J. Numer. Anal., {43} (2005), 1256--1275.

\bibitem{Grebenkov}
D.~S.~Grebenkov,
 \newblock\emph{Partially reflected {B}rownian motion: a stochastic approach
              to transport phenomena},
\newblock in ``Focus on probability theory", Nova Sci. Publ., New York (2006), 135--169.

\bibitem{Haberman}
B.~Haberman and D.~Tataru.
\newblock\emph{Uniqueness in {C}alder\'on's problem with {L}ipschitz conductivities},
\newblock Duke Math. J., \textbf{162} (2013), 496--516.

\bibitem{KaipioSomersalo}
J. Kaipio and E. Somersalo,
\newblock ``Statistical and computational inverse problems",
\newblock Springer-Verlag, New York, 2005.
   
\bibitem{Kaipioetal}
J.~P. Kaipio, V. Kolehmainen, E. Somersalo and M. Vauhkonen,
\newblock\emph{Statistical inversion and {M}onte {C}arlo sampling methods in electrical impedance tomography},
\newblock Inverse Problems {16} (2000), 1487--1522.

\bibitem{Karatzas}
I. Karatzas and S.~E. Shreve,     
\newblock ``Brownian motion and stochastic calculus",
 \newblock Springer-Verlag, New York, 1991.
   
\bibitem{Kolehmainen}
V. Kolehmainen, M. Lassas and P. Ola, 
\newblock\emph{Electrical impedance tomography problem with inaccurately known boundary and contact impedances}, 
\newblock IEEE Transactions on Medical Imaging, {27} (2008), 1404--1414. 

\bibitem{Isaacson}
R.~Kulkarni, G.~Boverman, D.~Isaacson, G.~J.~Saulnier, T.-J. Kao and J.~C.~Newell,
\newblock\emph{An analytical layered forward model for breasts in electrical impedance tomography},
\newblock Physiol. Meas. {29} (2008), 27--40.

\bibitem{LejayMaire}
A. Lejay and S. Maire,
 \newblock\emph{New {M}onte {C}arlo schemes for simulating diffusions in discontinuous media},
 \newblock J. Comput. Appl. Math. {245} (2013), 97--116.

\bibitem{Lejay}
A. Lejay and G. Pichot,
\newblock\emph{Simulating diffusion processes in discontinuous media: a numerical scheme with constant time steps},
\newblock J. Comput. Phys. {231} (2012), 7299--7314.

\bibitem{Lejay1}
A. Lejay,
\newblock\emph{Simulation of a stochastic process in a discontinuous layered medium},
 \newblock Electron. Commun. Probab. {16} (2011), 764--774.

\bibitem{LejayMaire1}
A. Lejay and S. Maire,
 \newblock\emph{Simulating diffusions with piecewise constant coefficients using a kinetic approximation},
\newblock Comput. Methods Appl. Mech. Engrg. {199} (2010), 2014--2023.

\bibitem{Lepingle}
D. L{\'e}pingle,
\newblock\emph{Euler scheme for reflected stochastic differential equations},
\newblock Math. Comput. Simulation, {38} (1995), 119--126.

\bibitem{MaireTanre}
S. Maire and E. Tanr\'{e},
\newblock\emph{Monte Carlo approximations of the Neumann problem}
\newblock Monte Carlo Methods Appl., {19} (2013), 201--236.

\bibitem{MaireNguyen}
S. Maire and G. Nguyen,
\newblock\emph{Stochastic finite differences for elliptic diffusion equations in stratified domains}
\newblock Preprint, Hal-00809203 (2013).

\bibitem{Mascagni0}
M. Mascagni and C. Hwang,
\newblock\emph{{$\ep$}-shell error analysis for ``walk on spheres'' algorithms},
\newblock Math. Comput. Simulation, {63} (2003), 93--104.

\bibitem{Mascagni}
M. Mascagni and N.~A. Simonov,
\newblock\emph{Monte {C}arlo methods for calculating some physical properties of large molecules},
 \newblock SIAM J. Sci. Comput. {26} (2004), 339--357.

\bibitem{Mikhailov}
G.~A. Mikhailov,
\newblock ``Optimization of weighted {M}onte {C}arlo methods",
\newblock Springer-Verlag, Berlin, 1992.

\bibitem{Muller}
M.~E. Muller,
\newblock\emph{Some continuous {M}onte {C}arlo methods for the {D}irichlet problem},
\newblock Ann. Math. Statist. {27} (1956), 569--589.

\bibitem{Pidcock}
M.~K.~Pidcock, M.~Kuzuoglu and K.~Leblebicioglu
\newblock\emph{Analytic and semi-analytic solutions in electrical impedance tomography: I. Two-dimensional problems}, 
\newblock Physiol. Meas. {16} (1995), 77--90.

\bibitem{Piiroinen}
P.~Piiroinen and M. Simon.
\newblock\emph{From Feynman-Kac formulae to numerical stochastic homogenization in electrical impedance tomography},
\newblock submitted for publication (2014).

\bibitem{Pursiainen}
S. Pursiainen,
\newblock\emph{Two-stage reconstruction of a circular anomaly in electrical impedance tomography},
 \newblock Inverse Problems {22} (2006), 1689--1703.

\bibitem{Revuz}
D.~Revuz and M.~Yor,
\newblock ``Continuous martingales and {B}rownian motion",
\newblock Springer-Verlag, Berlin, 1999.
    
\bibitem{Sabelfeld}
K.~K. Sabelfeld,
 \newblock ``Monte {C}arlo methods in boundary value problems",
 \newblock Springer-Verlag, Berlin, 1991.

\bibitem{SabelfeldTalay}
K.~K. Sabelfeld and D. Talay,
\newblock\emph{Integral formulation of the boundary value problems and the method of random walk on spheres},
\newblock Monte Carlo Methods Appl. {1} (1995), 1--34.

\bibitem{Simon}
M.~Simon.
\newblock\emph{Bayesian anomaly detection in heterogeneous media with applications to geophysical tomography},
\newblock Inverse Problems 30 (2014), 114013.

\bibitem{Simon1}
M.~Simon.
\newblock ``Anomaly detection in random heterogeneous media: Feynman-Kac formulae, stochastic homogenization and statistical inversion",
\newblock Ph.D thesis, University of Mainz, 2014.

\bibitem{Somersalo} 
E. Somersalo, M. Cheney and D. Isaacson,      
     \newblock \emph{Existence and uniqueness for electrode models for electric current computed tomography},
     \newblock SIAM J. Appl. Math., {52} (1992), 1023--1040.

\bibitem{Vaukhonen}
M.~Vauhkonen, W.~R.~B.~Lionheart, L.~M.~Heikkinen, P.~J.~Vauhkonen and J.~P.~Kaipio. 
\newblock\emph{A MATLAB package for the EIDORS project to reconstruct two-dimensional EIT images}
\newblock{Physiol. Meas.}, 22 (2001), 107--111.
\end{thebibliography}


\end{document}